\documentclass[10pt,a4paper,twoside]{amsart} 
\usepackage[utf8]{inputenc}
\usepackage[T1]{fontenc}
\usepackage[english]{babel}
\usepackage{microtype}
\usepackage{amsmath,amsthm,amssymb,mathtools,mathrsfs}

\usepackage{braket}
\usepackage{hyperref}
\usepackage[noabbrev,capitalize]{cleveref}
\usepackage{latexsym}
\usepackage{tikz}
\usepackage{tikz-cd}
\usepackage{float}
\usetikzlibrary{decorations.markings, arrows.meta, calc}

\tikzset{every picture/.style={line width=0.75pt}} %set default line width to 0.75pt 

	\hypersetup{hidelinks}
	\usepackage{wrapfig}
	\usepackage{lipsum,yhmath}    
	\usepackage{booktabs}
	\usepackage{graphicx}
        \usepackage{float}
	\usepackage{indentfirst}
	\usepackage{quoting}
	\usepackage{amsmath,amssymb}
	\quotingsetup{font=small}
	\usepackage{listings}
	\usepackage{multirow}
	\usepackage[lighttt]{lmodern}
	\usepackage{mathrsfs}
	\usepackage{amsthm}
	\usepackage{enumitem}
	\usepackage{array}
	\usepackage{tabularx}
	\usepackage{tikz}
	\usepackage{imakeidx}
        \usepackage{subcaption}
	\usepackage{pgfplots}
	\usepackage{bbm}
	\usepackage{amssymb}
	\usepackage{comment}
	\usepackage{wasysym}
	
	\usepackage{pifont}
	\usepackage{tikz-cd}
	\usepackage{fontawesome}
	\usepackage{easymat}
	\usepackage{xcolor}
	\usepackage{stmaryrd}
	\usepackage{hyperref}
	\hypersetup{hidelinks}
	
\newcommand{\longleadsto}{\tikz[baseline=-0.5ex]{\draw[->,decorate,decoration={snake,amplitude=.4mm,segment length=2mm}] (0,0) -- (2,0);}}
\newcommand{\giusi}[1]{{\color{purple}{\texttt G$\iota$us$\iota$: #1}}}

\newcommand{\ZZ}{\mathbb Z}

\newcommand{\R}{\mathbb R}

\newcommand{\ga}{\gamma}
\newcommand{\tga}{\tilde{\gamma}}

\newcommand{\Ga}{\Gamma}
\newcommand{\tGa}{\widetilde{\Gamma}}
\newcommand{\tal}{\widetilde{\alpha}}
\newcommand{\tbe}{\widetilde{\beta}}

\newcommand{\Div}{\operatorname{Div}}

\newcommand{\dil}{\mathrm{dil}}
\newcommand{\wt}[1]{\widetilde{#1}}

\newcommand{\calR}{\mathcal{R}}
\newcommand{\rr}{\operatorname{r}}

\DeclareMathOperator{\Pic}{Pic}
\DeclareMathOperator{\Jac}{Jac}

\DeclareMathOperator{\Prym}{Prym}
\DeclareMathOperator{\Id}{Id}
\DeclareMathOperator{\Ker}{Ker}

\DeclareMathOperator{\Hom}{Hom}

\let\Im\relax
\DeclareMathOperator{\Im}{Im}

 \theoremstyle{definition}
 \newtheorem{defn}{Definition}
 \newtheorem{construction}{Construction}
 
 \newtheorem{notation}{Notation}
  \newtheorem{lem}{Lemma}
   \newtheorem{ex}{Example}
  
 \newtheorem{rmk}{Remark}
 \theoremstyle{plain}
 \newtheorem{thm}{Theorem}[section]
 \newtheorem{prop}[thm]{Proposition}

\newtheorem{maintheorem}{Theorem}	
\DeclareMathOperator{\rk}{rk}
\newcommand{\La}{\Lambda}
\newcommand{\Si}{\Sigma}

\title{Continuity of the tropical Prym--Torelli map}
\author{Giusi Capobianco}
\address{Department of Mathematics\\ University of Roma Tor Vergata \\ 00133 Rome, Italy}
\email{\href{mailto:capobianco@axp.mat.uniroma2.it}{capobianco@axp.mat.uniroma2.it}}

\date{}

\begin{document}
\begin{abstract}
    We give a new intrinsic construction of the continuous tropical Prym variety, originally due to Röhrle and Zakahrov. In particular, given a harmonic double cover of graphs we explicitly construct the tropical Prym as integral torus. Moreover, we show that   the tropical Prym--Torelli map from the moduli space of harmonic double covers to the moduli space of principally polarised tropical abelian varieties is continuous. % In particular, this allows to construct a tropical universal Prym variety.
\end{abstract}
\thanks{}
\maketitle
\tableofcontents

\section{Introduction}
 Given  an étale double cover of smooth curves $f\colon \wt C\to C$, one can associate a principally polarised abelian variety  called the \emph{Prym variety} 
defined as the connected component of the identity of the kernel of the norm map
$\mathrm{Nm}_f: \Jac(\wt C) \to \Jac(C),$ namely
\[
\mathrm{Prym}(f) := \ker(\mathrm{Nm}_f)_0.
\]
Prym varieties can be thought of as generalisations of Jacobians
since the Jacobian locus is contained in the closure of the Prym locus, \cite{Wirtinger1895,BorowkaOrtega2022}.
The Prym-theoretic analog of the Torelli morphism is the so-called Prym--Torelli map  
\[
P_g\colon \mathcal R_g\to\mathcal A_{g-1},
\]
 from the moduli space of double covers to the moduli of principally polarised abelian varieties assigning to an étale double cover $f\colon \widetilde C\to C$ its Prym variety $\mathrm{Prym}(f)$.
%It is well known that the Prym--Torelli map is generically finite for $g\geq 6$, (\cite[Wirtinger theorem]{Beauville1977Prym}).
%An important question regards the injectivity of this map, first studied by Beauville in \cite{Beauville1977Prym} proving that the map is not globally injective. Later,  Donagi in \cite{donagi}, using the tetragonal construction, provided examples of distinct double covers with isomorphic Prym varieties in any genus. Nonetheless, \cite{Kanev1982Prym} and \cite{FriedmanSmith1982} independently proved that the map has generically degree one for $g\geq 7.$

In tropical geometry, the definition of  tropical Prym variety  was carried out by Jensen and Len in \cite{JensenLen_thetachars}.  Tropical Prym varieties arise   from harmonic morphisms of metric graphs, reflecting the classical construction associated with étale double covers of algebraic curves. 
 It was later shown that the Prym construction commutes with tropicalization \cite{LenUlirsch_Prym} and Prym varieties were
further investigated in
\cite{reiner2014critical, GZ24}.

However, it has been shown in \cite[Theorem 3.3, Example 3.6]{GZ24} that with this naive definition of the tropical Prym variety, the volume of the Prym
does not vary continuously in families.
More importantly, the tropical Prym--Torelli map $P_g^{trop}\colon \mathcal R^{trop}_g\to \mathcal A^{trop}_{g-1}$ is not continuous.

Therefore, in \cite{RZ_ngonal}, the authors modify the naive definition of the Prym variety that it is now referred to as $\Prym _d$, where $d$ stands for \emph{divisorial}, introducing the \emph{continuous} $\Prym$ variety  $\mathrm{Prym_c}$ and explore the relations between these two tropical abelian varieties in \cite[Section 4.4]{RZ_ngonal}.
The continuous Prym variety always carries a natural principal polarization and
comes with a natural isogeny $\mathrm{Prym}_c\to\mathrm{Prym}_d$ of degree $2^{d-1}$, where $d$ is the number of connected components of the dilation subgraph of the graph $\Gamma$. Under the assumption that the double cover $\pi\colon \tGa\to \Ga$ is either free or the dilation subgraph is connected, the authors show that the two definitions coincide, namely $\mathrm{Prym}_c(\pi)=\mathrm{Prym}_d(\pi)$.
The authors show that, within this setting, the volume of $\Prym_c$ varies continuously under the contraction of those edges that increase the genus of the dilated subgraph. 
%However, the authors do not address the modular continuity of the tropical Prym--Torelli map $\mathcal{R}^{trop}_g\to\mathcal A^{trop}_{g-1}$.
However,  since their result is purely tropical, they do not take into account the 
fact that the tropical Prym variety drops dimension when contracting multiple cycles, namely the existence of positive weights in the vertices of the tropical curve, that we call weighted metric graphs.

In this paper, we investigate this problem in the more general setting and give a novel and intrinsic construction for the continuous Prym variety of harmonic double covers. Given a double cover $\pi\colon \tGa\to\Ga$ of weighted metric graphs of genera $2g_{\Ga}-1$ and $g_{\Ga}$ and weights respectively $\tilde w$ and $w$, we define the \emph{extended} Prym variety the torus of dimension $g_{\Ga}-1$ obtained by adding to $\mathrm{Prym_c}(\pi)$, which has dimension $b_{1}(\tGa)-b_1(\Ga)$, a total of $\sum \tilde w-\sum w$ dimensions. As a consequence, we provide a proof of the continuity of the tropical Prym--Torelli map, a result that, to the best of our knowledge, does not yet appear in the literature although it is known to the experts.

%We relate this  construction to the already existing constructions of bases for anti-symmetric cycles  in \cite{LZ22} and in \cite{RZ_ngonal}. In \cite{LZ22}, the authors give a construction for the basis for anti-symmetric cycles for free double covers by starting from a spanning tree of $\Ga$ and choosing a specific spanning tree for $\tGa$, while in \cite{RZ_ngonal} the authors provide a basis for the lattice that defines the continuous Prym varietyby degenerating the basis obtained for free double covers following Construction A in \cite{LZ22}. Our construction generalises both and provides a basis for that lattice that is compatible with the contraction of edges.

%by exhibiting a specific basis that spans the lattice defining the continuous Prym variety introduced in \cite{RZ_ngonal}. In particular,in \cref{sec:continuity}, 

\begin{maintheorem}The tropical Prym--Torelli map $P^{trop}_g\colon \mathcal R^{trop}_ {g+1} \to\mathcal A^{trop}_g $, assigning to a harmonic double cover its extended continuous Prym variety, is continuous. \end{maintheorem}

In order to prove this theorem, we give a basis for anti-symmetric cycles associated to a double cover $\pi\colon\tGa\to \Ga$ initiated by Melo and Len for which it is more amenable to proving that the Prym variety is continuous under edge contractions. 
We relate this construction with the already existing  constructions in \cite{LZ22,RZ_ngonal}. 
While the authors in  \cite{LZ22} construct the basis for anti-symmetric cycles for the divisorial Prym variety by starting from a spanning tree of $\Ga$ and choosing a specific spanning tree for $\tGa$, in \cite{RZ_ngonal} the authors provide bases for the lattices that define both the divisorial and the continuous Prym variety
by degenerating the basis obtained for free double covers in \cite{LZ22}.
    The idea behind our construction is to give more  freedom to the choice of the spanning tree of the source graph $\tGa$ in order to provide a basis for the lattice that is compatible with the contraction of edges and that is intrinsic.

\begin{maintheorem}[Theorem \ref{thm:basis}]
    The  set of cycles given by Construction \ref{construction} is a basis for the lattice defining $\mathrm{Prym}_c$.
    \end{maintheorem}
    
    %Particular attention is given to loops of $\Ga$ such that the preimage consists of two parallel edges and we call them $\pi$-parallel loops and define $P_{\pi}$ the set of such edges in \cref{def:piparallelloops}.
    
In Section \ref{section1}, we introduce some definitions and fix the notations. Moreover, we prove some basic properties of harmonic morphisms of weighted metric graphs involving their betti numbers and the weight formula.

In Section \ref{sec:continuity},
we describe the intrinsic construction of the basis for anti-symmetric cycles and prove that the cycles are independent and that generate the lattice defining the continuous Prym variety as integral torus.

In Section \ref{sec:prymcontinuous}, we prove that this choice of basis is continuous under degeneration and  we address the continuity of the Prym--Torelli map.

\subsubsection*{Acknowledgments}
I would like to thank Margarida Melo and Yoav Len for suggesting this problem and for giving me insights on the construction of the intrinsic basis. Thanks also to Felix Röhrle and Dmitry Zakharov for the useful discussions.
The author was partially supported from Deutsche Forschungsgemeinschaft (DFG) through
the Collaborative Research Centre TRR 326 “Geometry and Arithmetic of Uniformized
Structures” (project number 444845124).

\section{Preliminaries}\label{section1}
In this section we recall the notions and definitions to fix the notations and we prove basic properties of harmonic double covers of graphs involving vertices with positive weight. Furthermore, we recall the structures of tropical abelian varieties and  define the extended Prym variety.

\subsection{Harmonic morphisms of weighted metric graphs}
\label{harmonicmorph}
We recall the notion of harmonic morphism of weighted metric graphs by adapting the definitions of \cite{MC} and \cite{Cap14} to this more general case. Suppose $(G,w,l)$ and $(G',w',l')$ are loopless weighted models for metric graphs $\Ga$ and $\Ga'$ then a \emph{morphism of loopless weighted models} $\phi\colon(G,l)\to (G',l')$ is a map of sets 
\[
V(G)\cup E(G)\to V(G')\cup E(G')
\]
such that
\begin{itemize}
    \item[i)] $\phi(V(G))\subseteq V(G')$;
    \item[ii)] if $e=xy$ is an edge of $G$ and $\phi(e)\in V(G')$ then $\phi(x)=\phi(e)=\phi(y)$;
    \item[iii)] if $e=xy$ is an edge of $G$ and $\phi(e)\in E(G')$  then $\phi(e)$ is an edge between $\phi(x)$ and $\phi(y)$;
    \item[iv)] if $\phi(e)=e'$ then $l'(e')/l(e)$ is an integer.
\end{itemize}

We call $\mu_{\phi}(e)=l'(e')/l(e)$ the slope of this map at $e\in E(G)$.

An \emph{indexed morphism} is a morphism  $\phi$ enriched by the assignment,
for every $e\in E(G)$, of a non-negative integer, the  {\it index} of $\phi$ at $e$,
written $r_{\phi}(e)$, 
such that $r_{\phi}(e)=0$ if and only if $\phi(e)$ is a point  and 
it is called {\it simple} if $r_{\phi}(e)\leq 1$ for every $e\in E(G)$.

An indexed morphism is {\it pseudo-harmonic}
if for every $v\in V(G)$ there exists a number, $m_{\phi}(v)$,
such that for every $e'\in E_{\phi_V(v)}(G')$   we have
\begin{equation}
\label{imeq}
m_{\phi}(v)=\sum_{e\in E_v(G): \phi(e)=e'}r_{\phi}(e). 
\end{equation}
\
A pseudo-harmonic morphism   is {\it finite} if $m_{\phi}(v)\geq 1$ for every $v\in V(G)$.

A pseudo-harmonic  morphism is {\it harmonic} if 
 for every $v\in V(G)$ we have, writing  $v'=\phi(v)$,
 \begin{equation}
\label{RH}
\sum_{e\in E_v(G)}(r_{\phi}(e)-1) \leq 2\Bigr(m_{\phi}(v)-1+w(v)-m_{\phi}(v)w'(v')\Bigl).
\end{equation}

Let $\phi:(G,w,l)\to (G',w',l')$ be a pseudo-harmonic   morphism.
Then for every $e'\in E(G')$ and $v'\in V(G')$ we  define the degree of $\phi$ and  the {\it ramification divisor} $R_{\phi}$ as follows 
\begin{equation}\label{degree}
\begin{aligned}
   & \deg \phi = \sum_{e\in E (G): \phi(e)=e'}r_{\phi}(e)= \sum_{v\in \phi^{-1}(v')}m_{\phi}(v), \\
&
R_{\phi}=\sum_{v\in V(G)} \Bigr(2\bigr(m_{\phi}(v)-1+w(v)-m_{\phi}(v)w'(v')\bigl)-\sum_{e\in E_v(G)}(r_{\phi}(e)-1)  \Bigl)v.
\end{aligned}
\end{equation}

\begin{rmk}
    The degree $\deg \phi$ does not depend on the choice of $e'\in E(G')$ (extension of the proof of Lemma 2.4, \cite{BN09}).
\end{rmk}

Finally, we define a  \emph{morphism of weighted metric graphs} $(\Gamma,w)$ and $(\Gamma',w')$  to be a continuous map $\widetilde\phi\colon \Ga\to \Ga'$ which is induced from a morphism $\phi\colon (G,w,l)\to (G',w',l')$ of any loopless weighted models with slope equal to $r_{\phi}(e)$ and $\deg\widetilde\phi=\deg\phi$,
 and we write either $\Gamma\to \Gamma'$ or $(\Gamma,w)\to (\Gamma',w')$ depending on the situation. We will simply denote by $\phi$ the morphism of weighted metric graphs. 
 For a point $x\in \Ga$, we define 
 \[m_\phi(x)=\begin{cases}
     m_{\phi}(v) & \text{if}\,\,x=v\in V(\Ga),\\
     r_{\phi}(e)& \text{if}\,\,x\in e\in E(\Ga)
 \end{cases}\]
 the \emph{multiplicity} of $\phi$ in $x$.
 The morphism $\widetilde\phi$ is \emph{harmonic} if $\phi$ is harmonic.

An \emph{isomorphism} between weighted metric graphs is an isometry which does not depend on the choice of the model and respects the weight function.% An \emph{isomorphism}  of double covers is an isomorphism $\phi\colon\Gamma_1\to \Gamma_2$ such that the diagram\[\begin{tikzcd}\Gamma_1 \arrow[r, "\phi"] \arrow[d, "\pi_1"] & \Gamma_2 \arrow[d, "\pi_2"] \\\Gamma \arrow[r, "Id"]                        & \Gamma                     \end{tikzcd}\]commutes.

%\begin{defn}\label{def:injectivity} Let $\phi:\Gamma\to\Gamma'$ be a morphism of metric graphs.  We say that $\phi$ is \emph{globally injective} at a point $p$ if it is the only point of $\Gamma$ mapping to $\phi(p)$. It is globally injective at a subgraph $\Gamma_0$  if it is globally injective at every point of the subgraph (namely, if $p$ is in $\Gamma_0$ and $q$ is any point in $\Gamma\setminus\{p\}$, then $\phi(q)\neq \phi(p)$). \end{defn}

%CONTRACTION BRIDGES
%\begin{rmk}    As in \cite[Examples 3.1 and 3.6]{BN07}, contracting bridges is \text{not} harmonic in general. These examples on combinatorial graphs can be immediately extended to the metric case. It follows from the definition of horizontal multiplicity that, if a bridge $e$, such that $\Gamma\setminus e$ consists of two metric graphs $\Ga_1$ and $\Ga_2$ of positive genera, gets contracted, then the morphism is not harmonic, because there is at least one edge $e'$ adjacent to $\phi(x)$ with no preimages adjacent to $x$ since there was the bridge $e$ in between. An example of a harmonic morphism that contracts bridges is shown in \cref{fig:bridges}, instead a non harmonic morphism is the map $\Psi$ in Figure \ref{fig:counterexampleofdeg2}.\end{rmk}

%\begin{ex}	\giusi{Example of a non harmonic morphism.}\end{ex}
As in \cite{Cap14}, and extending the result to metric graphs, given a pseudo-harmonic morphism $\phi\colon(\Ga,w)\to(\Ga',w')$, the \emph{pullback} map on divisors is defined as follows
\[
(\phi^*(D'))(x)=m_{\phi}(x)\cdot D'(\phi(x)),
\]
for all $x\in \Ga$. Moreover,  for every $D\in \Div(\Ga)$,
\begin{equation}\label{inequalityrank}
    \rr_{\Ga'}(\phi_*(D))\geq \rr_{\Ga}(D),
\end{equation}
which is a trivial extension of the proof of
\cite[Corollary 4.11]{BN09}. We recall the Riemann--Hurwitz formula for metric graphs that relates the canonical divisors of $(\Ga,w)$ and $(\Ga',w')$. This is an extension of the proof of \cite[Prop. 2.5]{Cap14} to the metric case.

By choosing a model of $\Ga$, say $(G,l)$, for every $v\in V(G)$  the \emph{canonical divisor} is defined as $K_{(\Ga,w)}(v):=2w(v)-2+\mathrm{val}(v)$,

 \begin{prop}[Riemann--Hurwitz]\label{RHweighted}
Let $\phi:(\Ga,w)\to (\Ga',w')$ be a pseudo-harmonic morphism of weighted metric graphs and let $(G,l)$ and $(G',l')$ be two models of genus $g$ and $g'$ respectively.
Then
\begin{equation*}
\label{RHeq}
 K_{(G,w)}=\phi^*K_{(G',w')}+R_{\phi}.
\end{equation*}
  The morphism $\phi$ is harmonic if and only if $R_{\phi}\geq 0$ (equivalently  $2g -2\geq \deg \phi(2g'-2)$).
\end{prop} 

    \subsection{Properties of harmonic double covers}
   
In what follows, we  consider harmonic double covers of (possibly) weighted metric graphs as introduced above. We will often abuse the notation and write $V(\Ga)$ and $E(\Ga)$ instead of $V(G)$ and $E(G)$, for some model $(G,l)$ of the  metric graph $\Ga$. For weighted harmonic morphisms, we will sometimes drop the function $w$ and just write $\phi\colon \Ga\to \Ga'$, but every metric graph is assumed to be weighted.

\begin{defn}
An \emph{unramified harmonic} morphism $\phi\colon \Ga\to \Ga'$ is a harmonic morphism such that the ramification divisor $R_{\phi}$ as in Equation \ref{degree}, is $R_{\phi}=0$. Therefore, for all $v\in V(\Ga)$, if we denote by $v'=\phi(v)$, then the following equality holds.
\begin{equation}
\label{RH1}
\sum_{e\in E_v(\Ga)}(r_{\phi}(e)-1) = 2\Bigr(m_{\phi}(v)-1+w(v)-m_{\phi}(v)w'(v')\Bigl).
\end{equation}

    A \emph{double cover}  $\pi\colon \widetilde\Ga\to \Ga$ is an unramified non-simple harmonic morphisms of degree $2$.
\end{defn}

\begin{notation}
     For \emph{double} covers, 
  since the index $r_{\pi}(e)$ is $1$ or $2$ depending on whether the edge is dilated or not, we will denote by $e_d(\tilde v)$ the number of dilated edges adjacent to $\tilde v$,
        so that we have \[e_d(\tilde v):=\sum_{e\in E_{\tilde v}(\wt G)}(r_{\pi}(e)-1),\]
        where $E_{\tilde v}(\wt G)$ is the set of edges adjacent to the vertex $\tilde v$ in $\tGa$.
\end{notation}

%\begin{rmk}  The genus $\tilde{g}$ of the cover graph $\widetilde\Ga$ is given by $\tilde{g}=2(g-1)+1=2g-1$.\end{rmk}

\begin{defn}
    Let $\pi\colon (\tGa,\tilde w)\to (\Gamma,w)$ be a harmonic morphism of weighted metric graphs of degree $2$. We define the \emph{dilation cycle}  \[\Ga_{\dil}:=\{x\in \Ga:|\pi^{-1}(x)|=1\},\]
     the \emph{dilation index}  is
\[
d(\tGa/\Ga)=\begin{cases} \mbox{number of connected components of }\Ga_{\dil},& \mbox{if $\pi$ is dilated,}\\
1, & \mbox{if $\pi$ is free.}
\end{cases}
\]

\end{defn}

\begin{rmk}\label{structureofthedilsubgraph}
    The set $\Ga_{\dil}$ is a subgraph of $\Ga$ which is the union of isolated vertices and even subgraphs.
    Indeed, by continuity of the morphism, the vertices adjacent to dilated edges are dilated and by harmonicity we have that $val(\tilde v)=2k$ for all the vertices $\tilde v\in\pi^{-1}(\Ga_{\dil})$.
    The latter follows from the formula $e_d(\tilde v)=\sum_{e\in E_{\tilde v}(\widetilde\Ga)}(r_{\pi}(e)-1) = 2\Bigr(m_{\pi}(\tilde v)-1+\tilde w(\tilde v)-m_{\pi}(\tilde v)w(v)\Bigl)$.    
    Therefore, it is a well-defined disconnected subgraph of $\Ga$ and the vertices $v\in \Ga_{\dil}$ such that $e_d(\pi^{-1}(v))=0$ are precisely the isolated ones.
  
\end{rmk}

\begin{lem}\label{lem:weightsformula} 
   Let $\pi\colon \widetilde{\Gamma}\to \Gamma$ be a harmonic double cover of  metric graphs with weight functions $\tilde w$ and $w$ respectively. Let $ v\in V(\Gamma)$, then 

    \begin{itemize}
        \item if $ v\not \in \Gamma_{\dil}$,  $\tilde w(\tilde v^+)=\tilde w(\tilde v^-)=w(v)$, 
        \item if $v \in \Gamma_{\dil}$,  $\tilde w(\tilde v)=2w(v)-1+\frac{e_d(\tilde v)}{2}$.
    \end{itemize}
   
    \begin{proof}
        It follows from the harmonicity formula (\ref{RH}).
        %harmonicity formula: \begin{equation*}  \sum_{e\in E_v(G)}(r_{\phi}(e)-1) = 2\Bigr(m_{\phi}(v)-1+w(v)-m_{\phi}(v)w'(v')\Bigl). \end{equation*}
        Indeed,  
        Let $ v$ be a non dilated vertex, and let $\tilde v^+$ be one of the two preimages (the case $\tilde v^-$ is analogous), then $e_d(\tilde v^+)=0$ and the multiplicity $m_{\pi}(\tilde v^+)$ is equal to $1$. Therefore, $0=2(\tilde w(\tilde v^+)-w(v))$. If $  v\in\Gamma_{\dil}$ then $\tilde v$ is the only preimage and
        the multiplicity $m_{\pi}(\tilde v)$ is $2$,  we get
        \begin{equation*}
            \frac{e_d(\tilde v)}{2}=1+\tilde w(\tilde v)-2w(v).
        \end{equation*}
    \end{proof}
\end{lem}
\begin{rmk}\label{positiveweight}
    By \cref{lem:weightsformula}, if $ v\in \Ga_{\dil}$ is an isolated vertex, then $\tilde w(\tilde v)=2w(v)-1$, therefore
isolated dilated vertices must have positive weight.
\end{rmk}
\begin{rmk}\label{rmk:relationbetweenBettinum}
    Note that for double covers of weighted graphs, the genus  of a  graph $G$  is given by $g(G)=b_1(G)+|w|$. As a consequence of Lemma \ref{lem:weightsformula},  different choices of dilation subgraphs produce graphs $\tGa$ with different first Betti numbers. In contrast with the case of free double covers, the first Betti number is therefore not constant, but depends on the dilation cycle. Indeed, let $\pi\colon \widetilde{\Gamma}\to \Gamma$ be a double cover, then \[b_1(\wt G)+|\tilde w|=\tilde g=2g-1=2(b_1(G)+|w|)-1,\]
    Therefore $b_1(\wt G)-b_1(G)=b_1(G)+2|w|-|\tilde w|-1$, and
    by \cref{lem:weightsformula} and the Handshaking lemma, we obtain
    \begin{equation}\label{eq:relationbettinumbers}
        b_1(\wt G)-b_1(G)=b_1(G)+n_d-m_d-1,
    \end{equation}
    
    where $n_d$ denotes the number of dilated vertices and $m_d$ the number of dilated edges of $\Ga$. By calling $b_1(\Ga_{\dil})$ and $c(\Ga_{\dil})$ the connected components of the dilated subgraph, and by  $W_{\pi}$ the set of isolated vertices, we could rewrite  $ b_1(\wt G)-b_1(G)=b_1(G)-1+\#W_{\pi}-b_1(\Ga_{\dil})+c(\Ga_{\dil})$. %Alternatively, it corresponds to $\#E_{non dil}-\#V_{non dil}$.
\end{rmk}
\begin{ex}
Figure \ref{dilatedweighted} is an example of a dilated and weighted double cover of a genus $7$ graph with a dilated isolated vertex $v_2,$ a non-dilated weighted vertex $v_1$,  and a dilated subgraph of genus $3$ on the right.
The weight function on the graph $\tGa$ is given by \cref{lem:weightsformula}. Moreover, $n_d=3$ and $m_d=2$ and we can check that $b_1(\tGa)=4=2b_1(\Ga)+n_d-m_d-1$, as prescribed by \cref{eq:relationbettinumbers}.
\begin{figure}[H]
    \centering
\begin{tikzpicture}
\coordinate (-1) at (-2,0);
\coordinate (-1a) at (-2,2);
\coordinate (-1b) at (-2,3);
\coordinate (0) at (-1,0);
\coordinate (0a) at (-1,2.5);

\coordinate (1) at (0,0); %leftvertex
 \coordinate (a1) at (0,2); %preimage - of the vertex (1)
\coordinate (b1) at (0,3); %preimage + of the vertex (1)
\coordinate (2) at (1,0); %rightvertex
 \coordinate (a2) at (1,2); %preimage - of the vertex (2)
 \coordinate (b2) at (1,3); %preimage + of the vertex (2)
 
 \coordinate (3) at (2,0);
 \coordinate (a3) at (2,2.5);

 \coordinate (4) at (3,0);
  \coordinate (a4) at (3,2.5);
  
 \coordinate (c1) at (0.4,2.4); 
 \coordinate (c2) at (0.6,2.6);
 %contruction of the edge in the middle: c1 and c2 are the middle points so that the segment is cutted in half

 \foreach \i in {2,1,a2,a1,b2,b1}
 \draw[fill=black](\i) circle (0.15em);
 \draw[black,fill=black](3) circle (0.25em);
 \draw[black,fill=black] (a3) circle (0.25em);
\draw[line width=1.5pt,black](0) circle (0.50em);
\draw[line width=1.5pt,black](0a) circle (0.50em);
\draw[line width=1.5pt,black](4) circle (0.50em);
\draw[line width=1.5pt,black](a4) circle (0.50em);
\draw[thin,black](-1) circle (0.50em);
\draw[thin,black](-1a) circle (0.50em);
\draw[thin,black](-1b) circle (0.50em);

\draw[thin, black] (-0.8,0)--(1);
\draw[thin, black] (-1.2,0)--(-1.8,0);
 \draw[thin, black] (1) to [out=90, in=90] (2);
\draw[thin, black] (1) to [out=90+180, in=-90] (2);

\draw[line width=1.5pt, black] (3) to [out=45, in=90+15] (2.8,0);
\draw[line width=1.5pt, black] (3) to [out=-45, in=-105] (2.8,0);

 \draw[line width=1.5pt, black] (a3) to [out=45, in=90+15] (2.8,2.5);
\draw[line width=1.5pt, black] (a3) to [out=-45, in=-105] (2.8,2.5);
 
\draw[thin, black] (b1) to [out=90, in=90] (b2);
\draw[thin, black] (a1) to [out=90+180, in=-90] (a2);
\draw[thin, black] (a1)--(-0.8,2.5);
\draw[thin, black] (b1)--(-0.8,2.5);
\draw[thin, black] (-1.8,3)--(-1.2,2.5);
\draw[thin, black] (-1.8,2)--(-1.2,2.5);
  \draw[thin, black] (b1)--(a2);
\draw[thin, black] (a1)--(c1);
    \draw[thin, black] (a3)--(a2);
    \draw[thin, black] (a3)--(b2);
    \draw[thin, black] (2)--(3);
    \draw[thin, black] (b2)--(c2);
  \draw (0.5,1.3) node [below] {$\downarrow$};
   \draw (-2,-0.2) node [below] {$v_1$};
    \draw (-1) node  {$2$};
    \draw (-1a) node  {$2$};
     \draw (-1b) node  {$2$};
    \draw (0) node  {$1$};
    \draw (-1,-0.2) node [below]  {$v_2$};
     \draw (0a) node  {$1$};
  \draw (4) node  {$2$};
   \draw (a4) node  {$4$};

  \end{tikzpicture}
 \caption{A dilated double cover with dilated and non-dilated vertices of positive weight. The thickness represents respectively the dilation cycles $\widetilde\Gamma_{\dil}\subset \tGa$ and $\Ga_{\dil}\subset \Ga$ made of an isolated vertex and a genus $3$ even subgraph.}
    \label{dilatedweighted}
\end{figure}
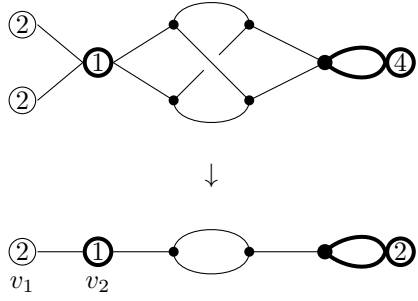
\end{ex}

\subsection{Tropical abelian varieties}

We quickly remind the reader about the structure of tropical abelian varieties. See \cite[Section $2$]{LZ22} and \cite[Section $4$]{RZ_ngonal} for more details. Let $\Lambda$ and $\Lambda'$ be finitely generated free abelian groups of the same rank and let $[\cdot, \cdot] : \Lambda \times \Lambda' \to \R$ be a non-degenerate pairing. The triple $(\Lambda, \Lambda', [\cdot, \cdot])$ defines a \emph{real torus with integral structure} (or simply an \emph{integral torus})  
\[
\Si=\Hom(\Lambda, \R) / \Lambda',
\]
where the inclusion $\Lambda' \subseteq \Hom(\Lambda, \R)$ is given by $\lambda' \mapsto [\cdot,\lambda']$.
The \emph{dimension} of an integral torus is $\dim_{\R}\Si=\rk \La=\rk \La'$. 
A \emph{polarisation} on $\Sigma$ is a group homomorphism $\xi : \Lambda' \to \Lambda$ such that 
\[
(\cdot,\cdot) = [\xi(\cdot), \cdot] \colon \Lambda'_\R \times \Lambda'_\R \to \R
\]
is a symmetric and positive definite bilinear form. A polarisation is necessarily injective, and is called \emph{principal} if it is also bijective.

%The invariant factors $(a_1,\ldots,a_g)$ of the Smith normal form (where $a_i\geq 1$ and $a_i|a_{i+1}$ for $i=1,\ldots,g-1$) define the \emph{type} of a polarisation $\zeta$, and a polarisation is principal if and only if all $a_i=1$. 

A polarisation defines an isogeny $(\zeta, \zeta) : \Sigma \to \Sigma^\vee$ to the dual, which is an isomorphism if and only if the polarisation is principal.

Let $f : \Sigma_1 \to \Sigma_2$ be a finite homomorphism of integral tori. Given a polarisation $\zeta_2$ on $\Si_2$, we define a polarisation \[f^*(\zeta_2)=f^\# \circ \zeta_2 \circ f_\#\] on $\Sigma_1$, called the \emph{induced polarisation}. We note that a polarisation induced from a principal one could be non-principal.

An integral torus together with a principal polarisation is called \emph{principally polarised tropical abelian variety} or \emph{pptav} for short.

A \emph{homomorphism of integral tori} $f=(f^{\#},f_{\#})\colon (\Lambda_1,\Lambda_1',[\cdot,\cdot]_1)\to (\Lambda_2,\Lambda_2',[\cdot,\cdot]_2)$ consists of a pair of homomorphisms $f^{\#}\colon \Lambda_2\to \Lambda_1$
and $f_{\#}\colon \Lambda_1'\to \Lambda_2'$
satisfying the relation 
\[
[f^{\#}(\lambda_2),\lambda_1']_1=[\lambda_2,f_{\#}(\lambda_1')]_2
\]
for all $\lambda_1'\in\Lambda_1'$ and $\lambda_2\in \Lambda_2$.
A homomorphism $f=(f^{\#},f_{\#})$ is an \emph{isomorphism} if $f_{\#}$ and $f^{\#}$
are isomorphisms.

The quintessential example of a pptav is the  \emph{Jacobian} of a metric graph $\Ga$, given by
\[
\Si_1=\mathrm{Jac}(\Ga)=(\Omega^1(\Ga),H_1(\Ga,\ZZ),[\cdot,\cdot]_1),
\] 
where the pairing $[\cdot,\cdot]_1$ is given by integration along cycles and the principal polarisation is natural map identifying  $H_1(\Gamma,\ZZ)$ with $\Omega^1(\Gamma)$. %Note that some sources use $\Omega(\wt\Gamma,\ZZ)$ rather than $H^1(\wt\Gamma,\ZZ)$ but we don't require that perspective here. 

%The \emph{Jacobian} of a metric graph $\Ga$ is the pptav \[\Si_1=\mathrm{Jac}(\Ga)=(\Omega(\Ga,\ZZ)),H_1(\Ga,\ZZ)),[\cdot,\cdot]_1),\] of dimension $g_{\Ga}$, where the pairing $[\cdot,\cdot]_1$ is described as follows.Let $G$ be an oriented model of $\Ga$. Then the simplicial chain group $C_1(G,\mathbb Z)$ is the free abelian group on the edges of $G$ containing the simplicial homology group $H_1(G,\mathbb Z)$ and we call $C_1(\Ga,\mathbb Z)$ and $H_1(\Ga,\mathbb Z)$ the direct limits. There is a natural isomorphism $H_1(\Ga,\mathbb Z)\to \Omega^1_{\Ga}$ sending a cycle $\sum a_e e$ to the $1-$form $\sum a_e\, de$.The \emph{integration pairing} \[\begin{tabular}{cccc}	    $[\cdot,\cdot]\colon$&$ \Omega_{\Ga}^1(\Ga)\times C_1(\Ga,\mathbb Z) $ & $\longrightarrow$&$\mathbb R$  \\	&$(\omega,\gamma)=\bigl(\sum a_ede,\sum b_ee\bigr)$&$\mapsto$&$\int_{\gamma}\omega=\sum a_eb_el(e)$ \\\end{tabular}\]restricts to a non-degenerate pairing $[\cdot,\cdot]_1\colon\Omega_{\Ga}^1(\Ga)\times H_1(\Ga,\mathbb Z)\to \mathbb R$. 

We recall the extended definition of the Jacobian of a weighted metric graph, introduced in \cite[Section 5]{BrannMeloViviani}. Let $\Ga_w$ be a graph of genus $g=b_1(\Ga)+|w|$, then the Jacobian is the pptav 
\[
\Si_1=\mathrm{Jac}(\Ga_w)=(\overline{\Omega^1(\Ga_w))},\overline{H_1(\Ga_w,\ZZ))},\overline{[\cdot,\cdot]}_1),
\] 
Where  we denote by $\overline{\Omega^1(\Ga_w))}$ and $\overline{H_1(\Ga_w,\ZZ))}$
 the two lattices of rank $g$ obtained by adding $ |w|$ dimensions to the two lattices of rank $b_1(\Ga_w)$.  More precisely,
\[
\overline{\Omega^1(\Ga_w)}=\Omega^1(\Ga_w)\oplus\ZZ^{|w|}
\]
and
\[
\overline{H_1(\Ga_w,\ZZ)}=H_1(\Ga_w,\ZZ)\oplus \ZZ^{|w|},
\]
the extended lattices.
The integration pairing $\overline{[\cdot,\cdot]}_1$ extends $[\cdot,\cdot]_1$ to $0$ for the generators of $\ZZ^{|w|}$.

\subsubsection{Prym varieties}
The second most studied tropical abelian variety is the  Prym variety associated to a double cover $\pi\colon \tGa \to \Ga$. % given by\[ \Sigma_2=\Prym(\pi)=((\text{Coker}\,\pi^*)^{tf},\mathrm{Ker}\,\pi_*,[\cdot,\cdot]_2).\]
%Here, $\pi^*$ and $\pi_*$ are the pullback and pushforward on cohomology and homology respectively, and $(\text{Coker}\,\pi^*)^{tf}$ is the quotient of $(\text{Coker}\,\pi^*)$ by its torsion subgroup. The most convenient way of thinking of $\mathrm{Ker}\,\pi_*$ is as anti-symmetric cycles in $\wt\Gamma$, namely oriented cycles that flip sign under the action of the involution. The pairing $[\cdot,\cdot]_2$ is induced by the integration pairing on $\Jac(\tGa)$ and the principal polarisation is obtained from the principal polarisation on the Jacobian  divided by 2. The Prym variety is a pptav of dimension $g-1$. 
Recently, in \cite{RZ_ngonal} the authors have introduced another definition of the Prym variety, called 
   the \emph{continuous} Prym, that is slightly different from the previous one which is now referred to as the \emph{divisorial} Prym variety.
We will  recall the structures of both as integral tori.

The double cover $\pi\colon\tGa\to\Ga$ and the involution $\iota\colon\tGa\to\tGa$ induce the following maps 
\begin{equation}
	\begin{aligned}
		&\pi^\ast :& \Omega^1(\Ga) &\longrightarrow \Omega^1(\tGa) \\
		&&\sum a_e de &\longmapsto \sum a_e(d \tilde e^+ + d \tilde e^-)\\[1em]
		&\pi_\ast :& H_1(\wt\Gamma, \ZZ) &\longrightarrow H_1(\Gamma, \ZZ) \\
		&&\sum a_{\tilde e^{\pm}}\tilde e^{\pm} &\longmapsto\sum (a_{\tilde e^+} + a_{\tilde e^-}) e \\
	\end{aligned}
	\qquad \text{and} \qquad
	\begin{aligned}
		&\iota^*:&\Omega^1(\tGa) & \longrightarrow  \Omega^1(\tGa) \\
		&&\sum a_{\tilde e^{\pm}}d\tilde e^{\pm} &\longmapsto\sum a_{\tilde e^{\pm}}d\tilde e^{\mp}\\[1em]
		&\iota_*:&H_1(\wt \Gamma, \ZZ) & \longrightarrow H_1(\wt \Gamma, \ZZ)\\
		&&\sum a_{\tilde e^{\pm}}\tilde e^{\pm} &\longmapsto\sum a_{\tilde e^{\pm}}\tilde e^{\mp}\\
	\end{aligned}
	\label{eq:pushpull}
\end{equation}
The pair of maps $(\pi^*,\pi_*)$ defines the \emph{norm map}  $\pi: \Jac(\wt \Gamma) \to \Jac(\Gamma)$ which  is the pushforward map on divisors of degree zero under the identifications $\Pic_0(\tGa)=\Jac(\tGa)$ and $\Pic_0(\Ga)=\Jac(\Ga)$. The involution $\iota:\tGa\to \tGa$ induces an involution $\iota=(\iota^*,\iota_*):\Jac(\tGa)\to \Jac(\tGa)$, and the composition $\pi\circ (\Id-\iota):\Jac(\tGa)\to \Jac(\Ga)$ is the zero map.

 The \emph{divisorial Prym variety} of the double cover $\pi:\tGa\to\Ga$ is the connected component of the identity of the kernel of $\pi$:
	\begin{equation}
		\Prym_d(\pi) = (\mathrm{Ker} \,\pi)_0= \Big( (\mathrm{Coker} \pi^*)^\mathrm{tf}, \quad \mathrm{Ker}\, \pi_*, \quad [\cdot, \cdot]_P \Big).
	\end{equation}
	The \emph{continuous Prym variety} of the double cover $\pi:\tGa\to\Ga$ is the integral torus:
	\begin{equation}
		\Prym_c(\pi)=\Big( \Omega^1(\tGa)/\mathrm{Ker} (\Id-\iota^*), \quad \mathrm{Im} (\Id-\iota_*), \quad [\cdot, \cdot]_P \Big).
		\label{eq:Prymc}
	\end{equation}
	The pairing $[\cdot,\cdot]_P$ on both tori is induced by the integration pairing on $\Jac(\tGa)$. 
	The dimension of the Prym variety  is equal to $g_0=b_1(\tGa)-b_1(\Ga)$.

We refer to \cite[Construction B]{LZ22}   for the construction of the basis for cycles for free double covers.

Let us remind the construction in the dilated case following  \cite[Proposition 4.20]{RZ_ngonal}. We will change the notation from $g(\Ga)$ to $b_1(\Ga)$ since the authors assume $|w|=0$ throughout the paper. Let us recall the following invariants

\begin{equation*} 
	\begin{aligned}
		A &= b_1(\Ga)-m_{\dil}+n_{\dil}-d(\tGa/\Ga), \\
		B &= d(\tGa/\Ga)-1, \\
		C &= m_{\dil}-n_{\dil}+d(\tGa/\Ga).
	\end{aligned} 
\end{equation*}
It follows from \cref{eq:relationbettinumbers} that
$
A+B=b_1(\Ga)-m_{\dil}+n_{\dil}-1=b_1(\tGa)-b_1(\Ga).
$

%\begin{prop} \label{prop:dilatedbasis} Let $\pi:\tGa\to\Ga$ be a dilated double cover. Then there exists a basis $\al_1,\ldots,\al_A$, $\ga_1,\ldots,\ga_C$	of $H_1(\Ga,\ZZ)$ and a basis $\tal^{\pm}_1,\ldots,\tal^{\pm}_A$,  $\tbe_1,\ldots,\tbe_B$, $\tga_1,\ldots,\tga_C$	of $H_1(\tGa,\ZZ)$, such that	\begin{align*}		\iota_*(\tal^{\pm}_i) &= \tal^{\mp}_i,		& \pi_*(\tal^{\pm}_i) &= \al_i,		& \pi^*(\al_i) &= \tal^+_i+\tal^-_i,		& i&=1,\ldots,A, \\		\iota_*(\tbe_j) &= -\tbe_j,		& \pi_*(\tbe_j) &= 0,		&&& j&=1,\ldots,B,\\		\iota_*(\tga_k) &= \tga_k,		& \pi_*(\tga_k) &= \ga_k,		& \pi^*(\ga_k) &= 2\tga_k,		& k&=1,\ldots,C,\end{align*}	where A,B and C as before.\end{prop}
In order to give a basis, the authors first contract the dilation subgraph $\Ga_{\dil}$ into isolated dilated vertices $v_0',\dots,v'_{d-1}$ resulting a double cover $\tGa'\to \Ga'$, where $d$ is the number of connected components of $\Ga_{\dil}$, and then consider a second double cover $\wt \Ga''\to\Ga''$ obtained by replacing the dilated vertices with loops and covering the loops with two parallel edges. In this way, given a spanning tree $T''$ of $\Ga''$, the preimage $\pi^{-1}(T'')$ is always disconnected and the authors use the same insights as in the free case to construct the basis.

In the free case
 $\Prym_c(\pi)=\Prym_d(\pi)$ and the dimension is $g_0=b_1(\Ga)-1$.
%Furthermore, $\Im(\Id-\iota_*)$ is equal to $\Ker \pi_*$ and is already saturated in $H_1(\tGa,\ZZ)$, hence $\Prym_d(\pi)$ is the group-theoretic image of $\Id-\iota$.  Moreover, the number of connected components of the group-theoretic kernel of $\pi_*$ is the index of $\pi_* (H_1(\tGa,\ZZ) )$ in its saturation $\big(\pi_* (H_1(\tGa,\ZZ) ) \big)^{\sat}$ in $H_1(\Ga,\ZZ)$, and Proposition~\ref{prop:freebasis} shows that\[\big(\pi_* (H_1(\tGa,\ZZ) )\big)^{\sat} = H_1(\Ga,\ZZ),\qquad \big[H_1(\Ga,\ZZ):\pi_*(H_1(\tGa,\ZZ))\big]=2.\]Finally, the principal polarisation $d:H_1(\tGa,\ZZ)\to \Omega^1_{\tGa}(\tGa)$ induces the polarisation\begin{align*}    \zeta: \big\langle\tal^+_1-\tal^-_1,\ldots,\tal^+_{g-1}-\tal^-_{g-1} \big\rangle &\longrightarrow \big\langle[d\tal^+_1],\ldots,[d\tal^+_{g-1}] \big\rangle,\\    \tal^+_i-\tal^-_i &\longmapsto 2[d\tal^+_i]\end{align*}on $\Prym_c(\pi)=\Prym_d(\pi)$, which is twice the principal polarisation $\zeta_c:\tal^+_i-\tal^-_i\mapsto[d\tal^+_i]$.
For a dilated double cover $\pi:\tGa\to \Ga$ with dilation index $d(\tGa/\Ga)$, instead, %we have\begin{align*}\Ker (\Id-\iota^*) &= \big\langle d\tal^+_1+d\tal^-_1,\ldots,d\tal^+_A+d\tal^-_A,d\tga_1,\ldots,d\tga_C \big\rangle, \\\Im \pi^* &= \big\langle d\tal^+_1+d\tal^-_1,\ldots,d\tal^+_A+d\tal^-_A,2d\tga_1,\ldots,2d\tga_C \big\rangle,\end{align*}
 the first lattices are the same, while
the second lattices are distinct. Indeed, the second lattice of $\Prym_c(\pi)$ is
\begin{equation}
\label{eq:Prymcbasis2}
\Im (\Id-\iota_*)= \big\langle \tal^+_1-\tal^-_1,\ldots,\tal^+_A-\tal^-_A,2\tbe_1,\ldots,2\tbe_B \big\rangle,
\end{equation}
while the second lattice of $\Prym_d(\pi)$ is
\[
\Ker \pi_*=(\Im (\Id-\iota_*))^{\mathrm{sat}}= \big\langle \tal^+_1-\tal^-_1,\ldots,\tal^+_A-\tal^-_A,\tbe_1,\ldots,\tbe_B \big\rangle.
\]
Hence, $\Prym_d(\pi)$ is the image of $\Id-\iota$, and the natural map $\Prym_c(\pi)\to \Prym_d(\pi)$ is a free isogeny with geometric degree equal to the index of $\Im (\Id-\iota_*)$ in $\Ker \pi_*$, which is equal to $2^B=2^{d(\tGa/\Ga)-1}$.

%In the dilated case the map $\pi_*$ is surjective, therefore $\pi_*(H_1(\tGa,\ZZ))$ is saturated in its image and the group-theoretic kernel of $\pi$ has a single connected component.  On $\Prym_d(\pi)$, the induced polarisation $\Ker \pi_*\to (\Coker \pi^*)^{\mathrm{tf}}$ sends $\tal^+_i-\tal^-_i$ to $2[d\tal^+_i]$ and $\beta_j$ to $[d\beta_j]$, hence has type $(1^B,2^A)$. On the other hand, the induced polarisation $\zeta:\Im(\Id-\iota_*)\to \Omega^1_{\tGa}(\tGa) / \Ker(\Id - \iota^\ast)$ on $\Prym_c(\pi)$ is equal to $\zeta=2\zeta_c$, where $\zeta_c$ is the principal polarization\begin{equation}   \begin{aligned}    \zeta_c:\Im(\Id-\iota_*) &\longrightarrow \Omega^1_{\tGa}(\tGa) / \Ker(\Id - \iota^\ast), \\    \tal^+_i-\tal^-_i &\longmapsto [d\tal^+_i], &i&=1,\ldots,A, \\     \beta_j &\longmapsto [d\beta_j], &j&=1,\ldots,B.   \end{aligned}\label{eq:Prymcpp}\end{equation}

\begin{rmk}
Note that with that description of the bases, the first lattices are the same, while the second lattices have slighly different bases given by an adjustment factor of 2 on some cycles inside the $\Prym_c(\pi)$ making the degeneration continuous as depicted in \cref{ex:continuity}. However, this construction does not take into account the weights in the dilated vertices, therefore we need to extend the definitions.
\end{rmk}

\begin{defn}\label{def:extendedPrym}
Let $\pi_w\colon \tGa\to\Ga$ be a  (possibly) weighted harmonic double cover where $\Ga$ is a weighted metric graph of genus $g_{\Ga}=b_1(\Ga)+|w|$. Then, the Prym variety of the double cover is defined as the following extended pptav
\begin{equation}
\overline{\Prym_c(\pi_w)}:=\Big( \overline{\Omega^1(\tGa)/\mathrm{Ker} (\Id-\iota^*)}, \quad \overline{\mathrm{Im} (\Id-\iota_*)}, \quad \overline{[\cdot, \cdot]}_P \Big).
	\label{eq:Prymcextended}
\end{equation}
%where the two lattices have rank $g_{\Ga}-1$, after adding $|\tilde w|-|w|$ dimensions, %where $\tilde w$ is the weight function on $\tGa$, 
More precisely, 
\[
\overline{\Omega^1(\tGa)/\mathrm{Ker} (\Id-\iota^*)}= \Omega^1(\tGa)/\mathrm{Ker} (\Id-\iota^*)\oplus\ZZ^{|\tilde w|-|w|}
\]
and \[
\overline{\mathrm{Im} (\Id-\iota_*)}= \mathrm{Im} (\Id-\iota_*)\oplus\ZZ^{|\tilde w|-|w|},
\]
and
the pairing $\overline{[\cdot,\cdot]}_P$  is the extension by the integration pairing $[\cdot,\cdot]_P$ in the trivial way. With this additional part, the dimension of the Prym variety  is now equal to $g_{\tGa}-g_{\Ga}$.
\end{defn}
We will often denote it simply by $\overline{\Prym_c(\pi)}$ and forget the dependence on the weighted graphs, whenever it is clear from the context.

\begin{rmk}
	Suppose $\pi_0\colon\tGa_0\to \Ga_0$ is obtained from $\pi\colon \tGa\to \Ga$ by contracting a subgraph $S\subset\Ga$ and its preimage $\wt S\subset \tGa$. Then, the integration pairing  $\overline{[\cdot, \cdot]}_P $ on $\overline{\Prym_c(\pi_0)}$ is the following 
	\begin{equation}
		\overline{[e_1, e_2]}_P=
		\begin{cases}
			0&\text{if either } e_1\,\text{or}\,e_2\in\wt S\\
			 [e_1, e_2]_P&\text{otherwise}.
		\end{cases}
	\end{equation}
\end{rmk}

\begin{ex}\label{ex:continuity}
Consider the free double cover of the dumbell graph in \cref{fig:freecase},  together with an orientation on the graphs.
Following \cite[Construction B]{LZ22}, the basis of anti-symmetric cycles consists of 
only one such cycle depicted in red.

 Let us contract the two loops together with their preimages, as in Figure \ref{fig:degeneration}. The result is  a dilated double cover whose  dilation subgraph has two connected components. By considering the construction of the basis starting from this graph, the anti-symmetric cycle is given by the two red edges. However, it is not the limit of the cycle constructed in the free case. Instead, the new basis for the continuous Prym, as in \cref{eq:Prymcbasis2}, takes into account the multiplicity $2$ given by the contraction. %Nonetheless, the dimension of the continuous Prym is $1$, and if we add the sum of the weights in the dilated components we get the same dimension which is $1$ but in general the dimension will be different.
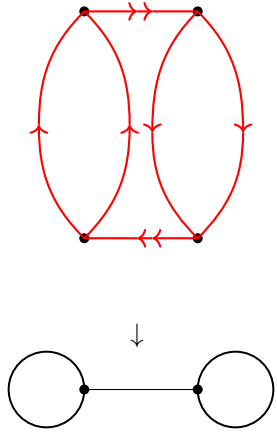
\begin{figure}[H]
	\centering
	
	\begin{tikzpicture}
		
		\coordinate (1x) at (-2.5,0);
		\coordinate (2x) at (2.5-4,0);
		\coordinate (2xleft) at (-5,0);
		\coordinate (3x) at (-4,0);
		\coordinate (a1x) at (-2.5,2);
		\coordinate (a3x) at (-4,2);
		\coordinate (b1x) at (-2.5,5);
		\coordinate (b3x) at (-4,5);

		\foreach \i in {1x,3x,a1x,a3x,b1x,b3x}
		\draw[fill=black](\i) circle (0.15em);
		\draw[thin, black] (3x)--(1x);

		\node[circle, draw, minimum size=1cm] (c) at (-0.5-4,0) {};
		\node[circle, draw, minimum size=1cm] (c1) at (2-4,0) {};

		\draw[thick, red,->] (b3x)--(-3.3,5);
		\draw[thick, red,->] (-3.3,5)--(-3.1,5);
		\draw[thick, red] (-3.1,5)--(b1x);

		\draw[thick, red,->] (a1x)--(-3.1,2);
		\draw[thick, red,->] (-3.1,2)--(-3.3,2);
		\draw[thick, red] (-3.3,2)--(a3x);

		\draw (0.7-4,1) node [below] {$\downarrow$};

		\draw[thick, red,->] (a3x) to [out=45, in=-90] (-3.4,3.5);
		\draw[thick, red] (-3.4,3.5) to [out=90, in=-45] (b3x);

		\draw[thick, red,->] (a3x) to [out=90+45, in=-90](-4.6,3.5);
		\draw[thick, red] (-4.6,3.5) to [out=90, in=180+45]  (b3x);
		
		\draw[thick, red,->] (b1x) to [out=-45, in=90] (-3.4+1.5,3.4);
		\draw[thick, red] (-3.4+1.5,3.4) to [out=-90, in=45] (a1x);
		
		\draw[thick, red,->] (b1x) to [out=180+45, in=90] (-4.6+1.5,3.4);
		\draw[thick, red] (-4.6+1.5,3.4) to [out=-90, in=90+45] (a1x);
		
	\end{tikzpicture}
	
	\caption{A free double cover of the dumbell graph. In red the anti-symmetric cycle $ \tal_1^+-\tal_1^-$ on $\wt \Ga$.}
	\label{fig:freecase}
\end{figure}
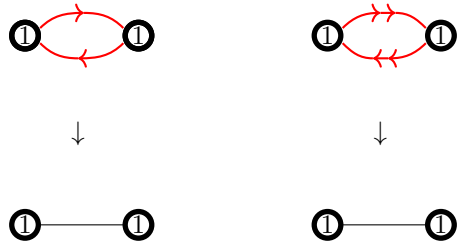
\begin{figure}[H]
	\centering
	
	\begin{tikzpicture}
		
		\coordinate (1x) at (1.5-4,0);
		\coordinate (2x) at (2.5-4,0);
		\coordinate (2xleft) at (-5,0);
		\coordinate (3x) at (-4,0);
		\coordinate (a1x) at (1.5-4,2.5);
		\coordinate (a3x) at (-4,2.5);
		\coordinate (b1x) at (1.5-4,2.5);
		\coordinate (b3x) at (-4,2.5);
		
		\coordinate (1) at (1.5,0);
		\coordinate (3) at (0,0);
		\coordinate (a1) at (1.5,2.5);
		\coordinate (a3) at (0,2.5);

		\foreach \i in {1,3,a1,a3,1x,3x,a1x,a3x,b1x,b3x}
		\draw[line width=2pt, 
		black](\i) circle (0.50em);

		\draw[thin, black] (-3.8,0)--(-2.7,0);
		\draw[thin, black] (-3.8+4,0)--(-2.7+4,0);
		
		\draw (0.7-4,1.5) node [below] {$\downarrow$};
		\draw (0.7,1.5) node [below] {$\downarrow$};

		\draw[thick, red,->] (-3.8,2.6) to [out=45, in=180] (-3.22,2.8);
		\draw[thick, red] (-3.22,2.8) to [out=0, in=90+45] (-2.7,2.6);

		\draw[thick, red,->] (-2.7,2.4)  to [out=180+45, in=0] (-3.3,2.2);
		\draw[thick, red] (-3.3,2.2) to [out=180, in=-45]  (-3.8,2.4);

		%\draw[thick, red] to [out=-45, in=0]  ;
		% \draw[thick, red,<-] (-3.22,2.2) to [out=180, in=180+45]  ;
		
		\draw[thick, red,->] (-3.8+4,2.6) to [out=45, in=180] (-3.3+4,2.8);
		\draw[thick, red,->] (-3.3+4,2.8) to [out=0, in=180] (-3.1+4,2.8);
		\draw[thick, red] (-3.1+4,2.8) to [out=0, in=90+45] (-2.7+4,2.6);

		\draw[thick, red,->] (-2.7+4,2.4)  to [out=180+45, in=0] (-3.2+4,2.2);
		\draw[thick, red,->]  (-3.2+4,2.2)  to [out=180, in=0]  (-3.4+4,2.2);
		\draw[thick, red] (-3.4+4,2.2) to [out=180, in=-45]  (-3.8+4,2.4);
		
		%\draw[thick, red] (-3.8+4,2.6) to [out=45, in=90+45] (-2.7+4,2.6);
		%	\draw[thick, red] (-3.8+4,2.4) to [out=-45, in=180+45] (-2.7+4,2.4);
		
		\draw (1x) node  {$1$};\draw (3x) node  {$1$};    \draw (a3x) node  {$1$};    \draw (b1x) node  {$1$};   \draw (1) node  {$1$};\draw (3) node  {$1$};    \draw (a3) node  {$1$};   \draw (a1) node  {$1$};

	\end{tikzpicture}
	
	\caption{Basis for anti-symmetric cycles respectively on the divisorial Prym on the left and the continuous Prym variety on the right hand side.}
	\label{fig:degeneration}
\end{figure}
\end{ex}

\subsection{The space of harmonic double covers and the tropical Prym--Torelli map}
The moduli space of harmonic double covers of graphs of a fixed degree $d$ was constructed by Cavalieri, Markwig and Ranganathan in \cite{cavmarkrang}. We will give an explicit description of this moduli space in the case of degree $2$, that we denote by $\mathcal R^{trop}_g$.

% Let $\Gamma$ be a tropical curve  and let $\gamma$ be a union of cycles and weighted vertices. Then an unramified double cover of $\Gamma$ with dilation cycle $\gamma$ will be given by a free double cover of $\Gamma\setminus \gamma$.

% *** Describe construction of all such double covers****

% **** Define morphisms of unramified double covers****

The cells of $\mathcal{R}^{trop}_g$ are indexed by harmonic double covers $f:\wt{G}\to G$. 
The cell $\mathcal R^{trop}_f$ can then be identified with $\R_{\geq 0}^n$ (where $n$ is the number of edges of $G$). Indeed,  each point $(\ell_1,\ldots,\ell_n)\in \R_{\geq 0}^n$ corresponds to 
choosing non-negative edge lengths  for the edges of $G$ and the corresponding edges of $\wt G$.
When some $\ell_i$ equals 0, we will eventually  identify the corresponding point of $\mathcal{R}^{trop}_g$ with the graph obtained by contracting the edge.  

The set of harmonic double covers of finite weighted graphs  forms a poset, where $f_0\preceq f$ whenever $f_0$ is obtained from $f$ by composing the following operations:
\begin{enumerate}
    \item Isomorphism of double covers; or

    \item Contracting edges of the target of $f$ and the corresponding edges of the source.
\end{enumerate}

  In the second case, the construction above yields a map  $\mathcal R^{trop}_{f_0}\to \mathcal R^{trop}_{f}$, identifying the former as a boundary component of the latter. 
The space $\mathcal R^{trop}_g$ is then the direct limit

\[
\mathcal R^{trop}_g = \lim_{\longrightarrow} \mathcal R^{trop}_{f}
\]
as $f$ varies over all the combinatorial types of edge-free double covers of stable graphs of genus $g$. 
This construction naturally identifies points corresponding to isomorphic double covers. Indeed, every automorphism of $f$ corresponds to a bijection from $\calR_{f}$ to itself, resulting in identifying points with their image. 

In this paper, we will consider the moduli space of principally polarised abelian varieties $\mathcal A^{trop}_g$ together with its Voronoi compactification $\mathcal A^{trop,V}_g$, see \cite{BrannMeloViviani, Chantorelli} for  a precise description of the cone decomposition.  

\begin{defn}\label{def:prymtorelli}
    The \emph{tropical Prym--Torelli map} $P^{trop}_g\colon\mathcal R^{trop}_g\to \mathcal A^{trop,V}_{g-1}$,
assigns to a harmonic double cover $\pi\colon \tGa\to \Ga$ its extended tropical Prym variety  $P^{trop}_g(\pi):=\overline{\Prym_c(\pi)}$.

\end{defn}

\begin{rmk}\label{rmk:quadraticform}
    Equivalently, one can define the map  using the positive semi-definite quadratic form associated to $\overline{\Prym_c(\pi)}$, for any double cover $\pi$. In \cite{BrannMeloViviani}, the authors use this perspective for the Torelli morphism and we refer to Section 5 for the details.
\end{rmk}

We will prove that this map is continuous in \cref{thm:prymtorellicontinuous}.

%We now consider a dilated double cover $\pi:\tGa\to \Ga$ with dilation index $d(\tGa/\Ga)$, so that $A=g_0-d(\tGa/\Ga)+1$ and $B=d(\tGa/\Ga)-1$. In terms of the bases we have\begin{align*}  \Ker (\Id-\iota^*) &= \big\langle d\tal^+_1+d\tal^-_1,\ldots,d\tal^+_A+d\tal^-_A,d\tga_1,\ldots,d\tga_C \big\rangle, \\ \mathrm{Im} \pi^* &= \big\langle d\tal^+_1+d\tal^-_1,\ldots,d\tal^+_A+d\tal^-_A,2d\tga_1,\ldots,2d\tga_C \big\rangle,\end{align*}hence the first lattices are the same and have basis\begin{equation} \label{eq:Prymcbasis1} (\Coker \pi^*)^{\mathrm{tf}}=\Omega^1_{\tGa}(\tGa)/\Ker (\Id-\iota^*)= \big\langle[d\tal^+_1],\ldots,[d\tal^+_A],[d\tbe_1],\ldots,[d\tbe_B] \big\rangle.\end{equation}The second lattices, however, are distinct. Indeed, the second lattice of $\Prym_c(\pi)$ \begin{equation}\mathrm{Im} (\Id-\iota_*)= \big\langle \tal^+_1-\tal^-_1,\ldots,\tal^+_A-\tal^-_A,2\tbe_1,\ldots,2\tbe_B \big\rangle,\end{equation}while the second lattice of $\Prym_d(\pi)$ is\[\Ker \pi_*=(\Im (\Id-\iota_*))^{\mathrm{sat}}= \big\langle \tal^+_1-\tal^-_1,\ldots,\tal^+_A-\tal^-_A,\tbe_1,\ldots,\tbe_B \big\rangle.\]Hence $\Prym_d(\pi)$ is the image of $\Id-\iota$, and the natural map $\Prym_c(\pi)\to \Prym_d(\pi)$ is a free isogeny with geometric degree equal to the index of $\Im (\Id-\iota_*)$ in $\Ker \pi_*$, which is equal to $2^B=2^{d(\tGa/\Ga)-1}$.

 \section{Construction of the intrinsic basis of anti-symmetric cycles}\label{sec:continuity}
%\MM{Dice anche qui un pò di più su che cosa in concreto sarà costruito per introdurre la scelta dell'albero. Devi assumere che un lettore può voler leggere l'inizio di ogno capitolo anche senza aver letto i precedenti}

In order to construct any basis for anti-symmetric cycles for the Prym variety associated to a double cover $\pi\colon \tGa\to \Ga$, one has to choose a spanning tree $T$ for $\Ga$ and a spanning tree $\wt T$ for $\tGa$ and construct the bases for $H_1(\Ga,\ZZ)$ and $H_1(\tGa,\ZZ)$. There are several ways to choose a spanning tree for the graph $\tGa$ by starting with a spanning tree $T$ in $\Ga$. For instance, in Construction A in \cite{LZ22}, for free double covers, the spanning tree $\wt T$ contains both preimages $\pi^{-1}(T)^+$ and $\pi^{-1}(T)^-$ and exactly one edge $\tilde e_g^+$ such that $ e_g\in E(\Ga\setminus T)$.
  Our construction of the spanning tree, which works  for both free and dilated covers, does not necessarily contain the whole of $\pi^{-1}(T)^-$. Instead, it
  allows to 
   choose which edges in $\pi^{-1}(T)^-$ should appear in $\wt T$.  
    The idea is that, starting with a spanning tree $T$ for $\Ga$, one may choose the spanning tree $\wt T$ for $\tGa$ 
   %used to form the basis of antisymmetric cycles 
   more generally in order to make it degenerate better under contractions. In detail, it
   must contain only the entire $\pi^{-1}(T)^+$ instead of $\pi^{-1}(T)^+\cup \pi^{-1}(T)^-$, and it must be contained in $\pi^{-1}(T)\cup \pi^{-1}(P_{\pi})^+$,
   where we label with $+$ and $-$ the two distinct preimages under the map $\pi$ and with $P_{\pi}$ the set of loops in $\Ga$ whose preimages are two parallel edges. 
   Furthermore, in the case of dilated double covers, our
   construction     differs from the approach used  in \cite{RZ_ngonal}, as it is more intrinsic.

\begin{construction}
	\label{construction}
	Let $\pi\colon\tGa\to \Ga$ be a harmonic double cover and consider an orientation on $\Ga$ and the induced orientation on $\tGa$. Consider a spanning tree $T\subset \Ga$ and its preimage $\pi^{-1}(T)\subset\tGa$. 
    Then $\pi^{-1}(T)=\pi^{-1}(T)^+\cup \pi^{-1}(T)^-$ and contains a spanning tree of $\tGa$ if and only if $\Ga_{\dil}\ne\emptyset$. Indeed, if $\pi$ is a free double cover, then the preimage of the tree is disconnected. Instead, if $\pi$ is dilated, then there is at least one vertex $v\in\Ga_{\dil}$ with $\tilde v^+=\tilde v^-$. Since by construction $\tilde v^+\in \pi^{-1}(T)^+$ and   $\tilde v^-\in \pi^{-1}(T)^-$, then $\pi^{-1}(T)$ is connected via $\tilde v^+=\tilde v^-$ and $\pi^{-1}(T)$ contains a spanning tree of $\tGa$. Moreover, if the number of connected components of $\Ga_{\dil}$ is greater than $1$, then
    $\pi^{-1}(T)$ contains a cycle that passes through two dilated vertices, therefore it strictly contains a spanning tree of $\tGa$.
    If $v$ is a dilated vertex, since the two preimages of $v$ coincide, we will denote by $\tilde v$ its unique preimage.
   In the discussion below, the case of loops will play an important role. In particular, consider the following type of loop:  
    \begin{defn}\label{def:piparallelloops}
Let $\pi\colon \tGa\to \Ga$ be a harmonic double cover of metric graphs.
         A loop $x\in\Ga$ is a \emph{$\pi$-parallel loop} if the preimage $\pi^{-1}(x)$ consists of two parallel edges $\tilde x^+,\tilde x^-$.

    \end{defn}

    Notice that $\pi$-parallel loops are not dilated and their preimage does not consist of two loops. 
    Moreover, the vertices adjacent to these loops are free, namely they have two distinct preimages.
   Let $\wt T$ be a spanning tree of $\wt \Gamma$ satisfying the following conditions:
    	\begin{itemize}
		 \item[(a)]\label{1.a} if $\pi$ is free and $\Gamma$ contains no $\pi$-parallel loops, then we will proceed as in construction A in \cite{LZ22} and consider a spanning tree $\wt T$ of $\wt \Gamma$ consisting of $\pi^{-1}(T)^+\cup \pi^{-1}(T)^-\cup \tilde{e}_g^+$, where $\{e_1,\dots,e_g\}=E(\Ga)\setminus T$ (notice that in this case $e_g$ can not be a loop of $\Gamma$);
		\item[(b)]\label{1.b} if $\pi$ is dilated or if $\Gamma$ contains $\pi$-parallel loops, then
        $\wt T$  must contain $\pi^{-1}(T)^+$ and must be contained in $\pi^{-1}(T)\cup \pi^{-1}(P_{\pi}%\{\text{$\pi$-parallel loops}\}
        )^+$.
    
	\end{itemize}

	We call such a tree $\wt T$ \emph{admissible} with respect to $T$.
	
	Following the notation of \cite[Prop. 4.14]{RZ_ngonal}, given a cycle $\gamma\in H_1(\Gamma, \mathbb Z)$ and  an oriented edge $e\in E(\Gamma)$, denote with $\langle\gamma,e\rangle$ the coefficient of $e$ in $\gamma$; analogously given a cycle $\gamma\in H_1(\tilde\Gamma, \mathbb Z)$ and an oriented edge $\tilde e\in E(\tGa)$, denote with $\langle\gamma,\tilde e\rangle$ the coefficient of $\tilde e$ in $\tilde\gamma$.

	Now, for each $e\in E(\Gamma)$, such that $\tilde e^-\not \in \wt T$,
	let $\wt\gamma_e^-$ be the unique cycle contained in $\wt T\cup \tilde e^-$ such that $\langle\wt\gamma_e^-,\tilde e^-\rangle=1$.
	Associated with $\wt\gamma_e^-$, consider the anti-symmetric cycle $$\beta_e:= \wt\gamma_e^- -\iota \wt\gamma_e^-.$$
	Note that $\beta_e=0$ if and only if $\wt\gamma_e^-=\iota\wt\gamma_e^-$, i.e., $\wt\gamma_e^-$ is symmetric. In the next section we will describe all situations for which we have $\beta_e=0$ and construct the basis for anti-symmetric cycles.
   \end{construction}

\subsection{Independence of the anti-symmetric cycles }

Let us discuss in both scenarios how many anti-symmetric cycles are non-zero and prove that  they   are all independent.
We start with \cref{1.a}.(a).

\begin{lem}\label{lem:independent(a)}
Let $\pi$ be a double cover as in \cref{construction}.(a). Then, there are $b_1(\Ga)-1$ non-zero antisymmetric cycles. Furthermore, they are all independent.
\end{lem}

\begin{proof}

         Suppose that $\pi$ has no dilation and $P_{\pi}=\emptyset$.
         In this case there are no edges $e$ in $T$ such that $\tilde e^-\not \in \wt T$,
         therefore the non-zero cycles $\beta_e$ satisfy $e\in \Ga\setminus T$.
         
	Then  $\wt T=\pi^{-1}(T)^+\cup \pi^{-1}(T)^-\cup \tilde{e}_g^+$ and $\beta_e= 0$ if and only if $e=e_g$.
    Indeed, if $e=e_g$ the cycle is of the form $\tga_{e_g}^-=\pi^{-1}(\ga_{e_g})^++\pi^{-1}(\ga_{e_g})^-$ where $\ga_{e_g}$ is the cycle in $\Ga$ that contains $e_g$ and it is clearly symmetric.
     If, instead, $e\in E(\Ga\setminus T)$ and $e\ne e_g$, %nor a $\pi$-parallel loop with preimage in $\wt T$, 
     then  $\tilde \ga_{e}^-$ contains $\tilde e^-$ and does not contain $\tilde e^+$ because $e\not\in T$, hence $\iota \tilde \ga_{e}^-$ contains $\tilde e^+$ and $\beta_e\ne 0$.
     Therefore
     there are exactly  $b_1(\Gamma)-1$ anti-symmetric  cycles $\beta_e$ that are non-zero, where $e\in E(\Gamma\setminus (T\cup e_g))$.
     Let us now  check that these $\beta_e$'s are independent. 
     We see that  these form a basis of anti-symmetric cycles in  $H_1(\wt \Gamma)$.
	In fact, the $\beta_e's$ are $b_1(\Gamma)-1$ independent anti-symmetric cycles in $\wt \Gamma$ since $\langle\beta_e,\tilde e^-\rangle=1$ and giving $f\neq e$, with $f\in E(\Gamma\setminus T)$, we have that $\langle\beta_f,\tilde e^-\rangle=0$.
	Notice that these cycles correspond, up to a sign given by orientation, to the ones in \cite[Construction B]{LZ22}.
\end{proof}

Let us focus now on
\cref{construction}.(b), when
either  $\pi$ is free and $\Ga$ contains $\pi$-parallel loops or $\pi$ is dilated.	
	Let $m_d$ and $n_d$ denote, respectively, the number of dilated edges and of dilated vertices of $\Gamma$. Denote also by $m_d^T$ the number of edges in $T$ which are dilated. Then in $T$ there are $|V(\Gamma)|-1-m_d^T$ edges which are not dilated. % Moreover, recall that $k$ is number of $\pi$-parallel loops $e$ such that $\tilde e^+\in \wt T$, namely $k=|P^{\wt T}_{\pi}|$ as in \cref{nota:piparallelloops}.
Moreover, let us introduce the following notation that will be used in the following Lemma and throughout the rest of the paper. 

\begin{notation}\label{nota:piparallelloops}
          We denote by $P_{\pi}$ the set of $\pi$-parallel loops of $\Ga$ and with $P^{\wt T}_{\pi}$ those $\pi$-parallel loops $x$ such that $\tilde x^+\in\wt T$ where $\wt T$ is an admissible spanning tree of $\tGa$ and we will use the term $k$ for the size of $P_{\pi}^{\wt T}$. Moreover, we will use the notation $V(P_{\pi})$ and $V(P_{\pi}^{\wt T})$ for the  sets of vertices attached to  $\pi-$parallel loops.   
        \end{notation}

\begin{lem}\label{lem:independent(b)}
    Let $\pi$ be a harmonic double cover as in \cref{construction}.(b). Then, there are $b_1(\wt\Gamma)-b_1(\Gamma)$ non-zero antisymmetric cycles. Furthermore, they are all independent.
\end{lem}
\begin{proof}
The anti-symmetric cycles $\beta_e$ in \cref{construction}.(b) arise both from edges in the complement of $T$ and from edges $e\in E(T)$ whose preimage $\tilde e^-$ does not belong to $\wt T$.
	We distinguish these two cases. Setting $k=|P_{\pi}^{\wt T}|$ as in \cref{nota:piparallelloops}, we show that the first set contributes $b_1(\Gamma)-(m_d-m_d^T)-k$ non-zero antisymmetric cycles, while the second contributes $-m_d^T+n_d-1+k$. Altogether, this yields
    \[b_1(\Gamma)-m_d+n_d-1=b_1(\wt\Gamma)-b_1(\Gamma)\]anti-symmetric cycles.
	\begin{itemize}
	    \item 
	Let $e\in E(\Gamma)\setminus T$ and assume that $e$ is not dilated. 
    Indeed, in the dilated case we have  $\tilde e^-=\iota\tilde e^-$, hence $\tilde\gamma^-_e=\iota \tilde\gamma^-_e$ and the associated cycle $\beta_e$ vanishes. For a non-dilated edge $e$, we associate  an anti-symmetric cycle $\beta_e$ as above. There are at most
	$
	b_1(\Gamma)-(m_d-m_d^T)
	$
	such cycles. Let us prove that  exactly $k$ of them are zero.
   If $e\in P^{\wt T}_{\pi}$, then $\tilde e^-+\tilde e^+$ forms a cycle so $\tilde \ga_{e}^-=\tilde e^-+\tilde e^+$ and therefore $\beta_{e}%\tilde \ga_{e}^--\iota \tilde \ga_{e}^-
   =0$. 
    Moreover, since  $\wt T\subset \pi^{-1}(T)\cup \pi^{-1}(P^{\wt T}_{\pi}%\{\text{$\pi$-parallel loops}\}
    )^+$,  the cycles $\beta_e$ corresponding to edges  $e\not\in P^{\wt T}_{\pi}$ are all non-zero, because $\tilde e^+\notin \wt\gamma_e^-$. 
Thus, so far we have constructed 
\[
b_1(\Gamma)-(m_d-m_d^T)-k
\]
non-zero anti-symmetric cycles each satisfying $\langle\beta_e,\tilde e^-\rangle =1$.

\item Let $e\in T$ such that $\tilde e^-\notin \wt T$.  In particular, the edge $e\not\in \Ga_{\dil}$, because otherwise $\tilde e^+=\tilde e^-$ and $\tilde e^+\in\wt T$. Moreover, such $e$ is not a $\pi$-parallel loop because $e\in T$.
	As before, let $\tilde\gamma_e^-$ be the unique cycle of $H_1(\tGa,\mathbb Z)$ contained in $\wt T\cup \tilde  e^-$ such that $\langle \tilde\gamma_e^-, \tilde e^-\rangle=1$.

	Notice that since $\wt \gamma_e^-\subset \pi^{-1}(T)\cup \pi^{-1}(P^{\wt T}_{\pi}%\{\text{$\pi$-parallel loops}\}
    )^+\cup \tilde e^-$ then
    $\wt\gamma_e^-$ contains $\tilde e^+=\iota(\tilde e^-)$ and so $\langle \wt\gamma_e^-, \tilde e^+\rangle=-1$. 
    Indeed, we are assuming that  $e\in T$ but $\tilde e^-\not \in T$, therefore $\pi^{-1}(T)^-$ is disconnected and $\pi^{-1}(T)^-\cup \tilde e^-$ is acyclic. In order to construct the cycle $\tilde\ga_{e}^-$ contained in $\wt T$ it is necessary to add either edges in $\pi^{-1}(T)^+$ or 
    $\pi^{-1}(P^{\wt T}_{\pi}%\{\text{$\pi$-parallel loops}\}
    )^+$. The latter, however, is not enough to complete the cycle, therefore there have to be edges in $\pi^{-1}(T)^+$ and the smallest cycle containing $\tilde e^-$ would also contain $\tilde e^+$ since there is only one path in $\pi^{-1}(T)^+$ connecting the endpoints of  $\tilde e^-$. The only edges traversing $\pi^{-1}(T)^+$ and $\pi^{-1}(T)^-$ are vertical edges attached to the vertices $\tilde v_i^+$ and $\tilde v_i^-$ corresponding to a $\pi$-parallel loop attached at $v_i$, which is either on the left or on the right side of the edge $\tilde e^-$.

    Now, if  the cycle $\wt\gamma_e^{-}$ is strictly contained in  $\pi^{-1}(T)$, it  passes through two intersection points between the trees $\wt T^{-}$ and $\wt T^{+}$. In particular, the cycle contains the unique path between the intersection points in each tree. As the trees are isomorphic, it follows that 
	$\iota \wt\gamma_e^-=-\wt\gamma_e^-$.
	Therefore, defining 
	\[
	\beta_e:=\wt\gamma_e^--\iota \wt\gamma_e^-
	\]
	as before, we in fact get that  $\beta_e=2\wt\gamma_e^-$. 
Otherwise,  if the cycle $\wt\gamma_e^{-}$ intersects exactly one preimage of a $\pi$-parallel loop that we call $x$, then 
\[
\beta_e=2E(\pi^{-1}(T)\cap \wt\gamma_e^{-})+\tilde x^+-\iota \tilde x^+.
\]
If the cycle $\wt\gamma_e^{-}$ intersects more than one preimage of $\pi$-parallel loops then it will traverse the two trees $\pi^{-1}(T)^+$ and $\pi^{-1}(T)^-$ at most twice (because it has to be a simple cycle) and will be of the form 
\[
\beta_e=2E(\pi^{-1}(T)\cap \wt\gamma_e^{-})+\tilde x_1^+-\iota \tilde x_1^++\tilde x_2^+-\iota \tilde x_2^+,
\]
where $x_1,x_2\in P^{\wt T}_{\pi}$.
In this situation, all edges in the cycle $\beta_e$ except the preimages of the edges in $P^{\wt T}_{\pi}$ will have multiplicity $2.$

    %Indeed, let $e\in T$ such that $\tilde e^-\not\in \wt T$. Then for any $f\ne e$, if $\tilde f^+\in \wt\gamma_e^-$ also $\tilde f^-$ is in  $ \wt\gamma_e^-$. And since $\pi(\wt\gamma_e^--\iota\wt\gamma_e^-)=0$ and sits inside $T$, there has to be for each edge also its conjugated by the involution $\iota$.
	Let us see that there are
    \begin{equation}\label{eq1}
        |E(\pi^{-1}(T)\setminus \wt T)|=-m_d^T+n_d-1+k
    \end{equation}
    
	of such cycles and these are all non-zero.

    This can be proved by induction on the cardinality of $P^{\wt T}_{\pi}$, namely on  $k\geq 0$.
If $k=0$, %there are no $\pi$-parallel loops with preimage in $\wt T$, 
the number of edges $e$ such that $\tilde e^- \not\in\wt T$ is obtained by taking the total number of edges in $\pi^{-1}(T)$ which is $2V(\Ga)-2$
and subtracting the number of edges in a spanning tree of $\tGa$ which is $|V(\tGa)|-1$.
Since each non-dilated vertex of $V(\Ga)$
has two preimages and each dilated
 dilated vertex has only one, we have
$
V(\tGa)=2(V(\Ga)-n_d)+n_d
$
Subtracting also  the number of dilated edges, we obtain:
\[2V(\Ga)-2-(2(V(\Ga)-n_d)+n_d-1)-m_d^T=-m_d^T+n_d-1.\]

Now, assume that \eqref{eq1} holds for $k-1$, and let us prove it for $k$.
Starting from a spanning tree $\wt T$ with $k-1$ preimages of $\pi$-parallel loops in $\wt T$, consider the subgraph $\wt T'=\wt T\cup \tilde x_k^+$,  where $x_k$ is a $\pi$-parallel loop  attached to a vertex $v_k$ in $T$. The graph $\wt T'$ contains a cycle, since the edge $\tilde x_k^+$ connects the  vertex $\tilde v_k^+$ with the vertex $\tilde v_k^-$, and these two vertices were already connected by a unique path in $\wt T$, as $\wt T$ is a spanning tree. 
To include $\tilde x_k^+$ in the spanning tree, one must remove an edge from $\wt T$. %and choose it such that it is not a $\pi$-parallel loop.
Choose $e_0\in T$ be such that $\tilde e_0^+$ and $\tilde e_0^-$ lie in $\wt T$ (not necessarily adjacent to $x_k$), and consider $\wt T'\setminus \tilde e_0^-$. Then the number of edges having only one preimage in the admissible spanning tree increases by one and  the edge $e_0$ becomes one of them.
Thus, adding the $k$-th $\pi$-parallel loop contributes exactly one additional such edge
and in total 
there are $k$  edges arising from the $k$ $\pi$-parallel loops in $P_{\pi}^{\wt T}$.
Summing up, we have constructed \[b_1(\Gamma)-m_d+n_d-1=b_1(\Gamma)+\sharp W_\pi -1 -b_1(\Ga_{\dil})+c(\Ga_{\dil})=b_1(\wt\Gamma)-b_1(\Gamma)\] cycles.
\end{itemize}

Let us check that these are all independent. 
	Indeed,  if  $e\in E( (\Ga\setminus T)\setminus P^{\wt T}_{\pi})$ then $\langle \beta_e,\tilde e^-\rangle=1$ and $\langle \beta_f,\tilde e^-\rangle=0$ for all $f\in E(\Gamma\setminus T)$ such that $f\neq e$. %if  $e\in P^{\wt T}_{\pi}$then  $\langle \beta_e,\tilde e^-\rangle=0$ and $\langle \beta_f,\tilde e^-\rangle=0$ for all $f\in E(\Gamma\setminus T)$ such that $f\neq e$,
    If $e\in T$ such that $\tilde e^-\not\in T$ then $\langle \beta_e,\tilde e^-\rangle=2$  and $\langle \beta_f,\tilde e^-\rangle=0$ for all $f\in E(\Gamma\setminus T)$ such that $f\neq e$. %and $f\ne \{\text{$\pi$-parallel loops}\}$. 
    In the first case we exclude $e\in P^{\wt T}_{\pi}$ because $\beta_e=0$ and in the second case
    $e\not\in P_{\pi}$ because $e\in T$ %otherwise $\beta_e=0$,
    therefore the intersection between edges in $\Ga\setminus T$ and   $\tilde e^-$ cannot be $1$ (it could happen when $\beta_e$ is not entirely $2\tga_e^-$).

\end{proof}

    Let us sum up the intersection pairing in the following

\begin{center}
\begin{tabular}{|c|c|c|}
\hline
 $\langle\cdot ,\cdot\rangle$& $\tilde e^-$ %s.t. $\tilde e^{\pm}\not \in \wt T$%\in E((\Ga\setminus T) \setminus P_{\pi}^{\wt T})$
 &$\tilde f^-$ s.t. $f\ne e$ %s.t. $\tilde f^{\pm}\not\in \wt T$%\in E((\Ga\setminus T) \setminus P_{\pi}^{\wt T})\ne e$ 
\\ \hline
1.(a): $\beta_e$ with $e\in E(\Ga\setminus T\cup e_g) $ & 1 & 0
 
  \\ \hline
1.(b):$\beta_e$ with $e\in E((\Ga\setminus T) \setminus P_{\pi}^{\wt T})$ & 1 & 0 \\ \hline

1.(b):$\beta_e$ with  $e\in E(T)$ and $\tilde e^-\not \in E(\wt T)$ & 2& 0  \\ \hline

\end{tabular}
\end{center}

We will make extensive use of the set of anti-symmetric cycles $\beta_e$ and we give it a name. Moreover, we will prove in \cref{thm:basis} that they actually form a basis for the lattice $\mathrm{Im}(\Id-\iota_*)$.
    
\begin{defn}\label{def:basis}
    Let $\pi\colon \wt \Ga\to \Ga$ be a harmonic double cover of metric graphs and let $\wt T$ be an admissible spanning tree with respect to a spanning tree $T$ of $\Ga$. Construct the antisymmetric cycles $\beta_e$ as in Construction \ref{construction}.(a) and \ref{construction}.(b).
    We denote with 
    $E_{nz}\subset E(\Ga)$
the set of edges $e$ that give rise to non-zero cycles $\beta_e$. 
    We define
    \[
    \mathcal{B}_{\wt T}:=\{\beta_e\}_{e\in E_{nz}},
    \]
    where
    
    \begin{equation}\label{setofedgesfornonzerocycles}
    E_{nz}=
    \begin{cases}
       E(\Ga\setminus T\cup e_g)  & \text{in}\quad (1.a)  \\
            E((\Ga\setminus T) \setminus P_{\pi}^{\wt T})\cup E^{\wt T}_{\pi}  & \text{in}\quad (1.b)
    \end{cases}
  \end{equation}
    
     where  $E^{\wt T}_{\pi}$ denotes the set of edges $f\in E(T)$ for which $\tilde f^+\in E(\wt T)$ but  $\tilde f^-\not \in E(\wt T)$. By Equation \ref{eq1}, the total number of such edges is $n_d-1+k$ and we label the elements of $E^{\wt T}_{\pi}$ by $f_i$, for $i=1,\dots,n_d-1+k$.
\end{defn}

\begin{rmk}
    From \cref{lem:independent(a)}
and \cref{lem:independent(b)}, the cardinality of the set $ \mathcal{B}_{\wt T}$ is
exactly 
\begin{equation*}
|\mathcal{B}_{\wt T}|=
        b_1(\tGa)-b_1(\Ga).
\end{equation*}
%\MM{ma nel caso free la formula e' comunque sempre quella di sotto, giusto? La differenza è che la differenza dei generi non è più il genere di quella di sotto meno 1}\giusi{Yes, ho riscritto.}
Indeed, the formula $ b_1(\tGa)-b_1(\Ga)=b_1(\Ga)-m_d+n_d-1$ shows that, in the case of a free  double cover, when   $m_d=n_d=0$,  the number of such anti-symmetric cycles concides with $b_1(\Ga)-1$ as in \cref{lem:independent(a)}.
\end{rmk}
%\MM{non si capisce il senso di questo remark. Comincia per prendere un cammino qualunque nel albero del grafo di sotto e chiarisci che cosa intendi dimostrare nei diversi casi. Mi sembra inoltre che vuoi mostrare una cosa importante per la dimostrazione sotto, quindi mi sembra che sarebbe meglio enunciarlo in un lemma e dimostrare}\giusi{magari faccio un esempio e poi nel lemma lo dimostro.}

\begin{ex} \label{ex:uniquepath}
Let us consider the harmonic double cover in \cref{ex:dilatedpaths}.
 An  admissible spanning tree is given by  $\wt T=\{\tilde f_1^+,\tilde f_0^+,\tilde f_5^+,\tilde f_2^+,\tilde x_1^+,\tilde f_3^+,\tilde f_4^+, \tilde f_1^-,\tilde f_5^-,\tilde f_4^-\}$. % where the edges $\tilde f_i^-$   that are not contained  in $\wt T$ are highlighted in red.
Note that the choice of admissible spanning tree determines different paths in the tree $T$ that contain the edges $f_i\in E^{\wt T}_{\pi}$, with $E^{\wt T}_{\pi}$ the set of edges of the spanning tree with only one preimage in $\wt T$ as in \cref{def:basis}. 
By \cref{eq1}, the set $E^{\wt T}_{\pi}$ has cardinality $|E^{\wt T}_{\pi}|=n_d-1+k$. In this case there are $3$ 
edges $f_i$ such that $\tilde f_i^-\not\in \wt T$, that are highlighted in red in the picture.

Observe that, there are paths with endpoints in $V(\Ga_{\dil})\cup V(P^{\wt T}_{\pi})$ that contain more than one edge of $E^{\wt T}_{\pi}$.   For instance, the path $\{f_1,f_2,f_3\}$ between the two dilated vertices, contains $2$ edges of that type. Similarly, the path $\{f_5,f_0,f_2,x_1\}$ between one dilated edge and one vertex with a loop $x_1\in P^{\wt T}_{\pi}$, contains again two edges of $E^{\wt T}_{\pi}$. On the other hand, one  can also find paths in $T$,  that  contain only one of  the edges $f_i$, for $i=1,2,3$, namely

 \[
 Q_1=\{f_5,f_0,f_1\}
 \]
  \[
 Q_2=\{f_1,f_2\}
 \]
  \[
 Q_3=\{f_3\}.
 \]

\end{ex}

In the following Lemma, we show that there is a unique path in $T$  with endpoints in $V(\Ga_{\dil})\cup V(P^{\wt T}_{\pi})$, that contains exactly one edge $f_i\in E^{\wt T}_{\pi}$.
As observed in the above example, %\cref{ex:uniquepath}, 
certain paths may contain more than one edge from $E^{\wt T}_{\pi}$.

\begin{lem}\label{rmk:edgeofthesecondkind}
    Let $\pi$ be as in Construction \ref{1.b}.(b) and suppose $m_d=0$.
    Then, for each $f_i\in E^{\wt T}_{\pi}$, there exists a unique  path $Q_i$ in $T$ with endpoints in $V(\Ga_{\dil})$ or in $V(P^{\wt T}_{\pi})$
    such that $Q_i\cap E^{\wt T}_{\pi}=\{f_i\}$.

    %there exist  paths $Q_i$ in $T$ between dilated vertices $u_0^{(i)},v_0^{(i)}$ or vertices attached to loops in $P^{\wt T}_{\pi}$  containing exactly one edge $f_i$ such that $\tilde f_i^+\in \wt T$ but $\tilde f_i^-\not\in\wt T$, with $i=1,\dots,n_d-1+k$.
\end{lem}

\begin{proof}
  
From \cref{eq1}, the number of edges in $E^{\wt T}_{\pi}$, namely the edges $f\in T$ such that $\tilde f^+\in\wt T$ but $\tilde f^-\not\in\wt T$,
is $n_d-1+k.$
When $k=|P^{\wt T}_{\pi}|=0$ and $n_d=2$, then $|E^{\wt T}_{\pi}|=1$
and there exists a unique path $Q$ in $T$ between the two dilated points. 
%The preimage $\pi^{-1}(Q)$ in $\tGa$ contains a cycle 
The path $Q$ will contain exactly one 
 edge in $E^{\wt T}_{\pi}$. Indeed, the graph $\wt\Gamma$ will contain two copies of the path attached at their endpoints, that are dilated vertices, and the construction of the admissible spanning tree $\wt T$ involves removing a single edge from this cycle. 
The same thing holds in the cases $(n_d,k)=(1,1)$ and $(n_d,k)=(0,2)$; i.e., there is only one path $Q$ in $T$ containing the unique edge in $E^{\wt T}_{\pi}$.
%\MM{non capisco la tua dimostrazione di quest'ultima parte. Però credo che capisco quello che intendi dire: è possibile considerare cammini in $T$ le cui estremità siano in $V(\Ga_{\dil})$ o in $V(P^{\wt T}_{\pi})$ e che non contengano nessun altro vertice di quel tipo al suo interno} \giusi{in realtà, come nell'esempio in   \cref{ex:dilatedpaths}, il cammino $\{e_2, f_1,f_2\}$ è tra un vertice dilatato e uno che ha il $\pi$-parallel loop $x_1$ e non ha nessun altro all'interno. Però quello che voglio dire io è che questo cammino poi si biforca a destra e sinistra e quello a sinistra ne contiene solo uno. Ho riscritto:} 

Suppose now $n_d-1+k>1$,  between two endpoints in $V(\Ga_{\dil})\cup V(P^{\wt T}_{\pi})$ the total number of paths  in $T$ is $\binom{n_d+k}{2}$, and each path contains at least one edge $f_i\in E^{\wt T}_{\pi}$, otherwise its preimage would create a cycle in $\wt T$.
We claim that we can find a unique path that contains exactly one edge of $E^{\wt T}_{\pi}$.
By Equation (\ref{eq1}), the preimage $\pi^{-1}(T)$ contains
$n_d-1+k$ simple and independent cycles $\alpha_i$ and $\wt T$ is obtained by choosing an edge $\tilde f_i^{-}\in E(\pi^{-1}(T)^-)$ to remove from $\alpha_i$, for all $i=1,\dots,n_d-1+k$. The image of this cycle $\pi(\alpha_i)$ is contained in $T\cup P^{\wt T}_{\pi}$ and 
determines a path $Q_i$ in $T$, whose endpoints are in $V(\Ga_{\dil})\cup V(P_{\pi}^{\wt T})$. This path contains the unique edge $f_i\in T$ such that $\tilde f_i^-\not\in \wt T$, namely   $f_i\in E^{\wt T}_{\pi}$.  
Indeed, if there was another edge 
$f_j\ne f_i \in Q_i\cap E^{\wt T}_{\pi}$,
then the preimage $\pi^{-1}(Q_i)$ would be a simple cycle with two edges removed, and $\wt T$ would not be a spanning tree anymore. %Finally, the restriction $Q_i=P_i\cap T$ is the unique path in $T$ and the result follows.

%and let $v_0^{(i_1)},v_0^{(i_2)}$ be either in $V(\Ga_{\dil})$ or in $V(P^{\wt T}_{\pi})$. Denote by $P_i$ the path in $T$ between two such vertices. There are $\binom{n_d+k}{2}$ paths $P_i$, and each one contains at least one edge $f_i\in E^{\wt T}_{\pi}$, otherwise the preimage $\pi^{-1}(P_i)$ would contain a cycle.

%If $P_i\cap P_j=\emptyset$ for all $i\ne j$ then, as before, each disjoint path contains exactly one edge $f_i$ and we call it $Q_i$. Otherwise, suppose that there exist $i\ne j$ such that $P_i$ is the path with endpoints $v_0^{(i_1)},v_0^{(i_2)}$ and $P_i=v_0^{(j_1)}-v_0^{(j_2)}$ and$P_i\cap P_j\subset E(T)\ne \emptyset$.Consider the paths with endpoints$v_0^{(i_1)}-v_0^{(i_2)}$, $v_0^{(i_1)}-v_0^{(j_1)}$, $v_0^{(i_1)}-v_0^{(j_2)}$, $v_0^{(i_2)}-v_0^{(j_1)}$, $v_0^{(i_2)}-v_0^{(j_2)}$ and $v_0^{(j_1)}-v_0^{(j_2)}$. There are $6$ paths and $3$ edges of that type, therefore among these paths there are exactly $3$ that we call $Q_i$ containing only one $f_i\in E^{\wt T}_{\pi}$.

\end{proof}

  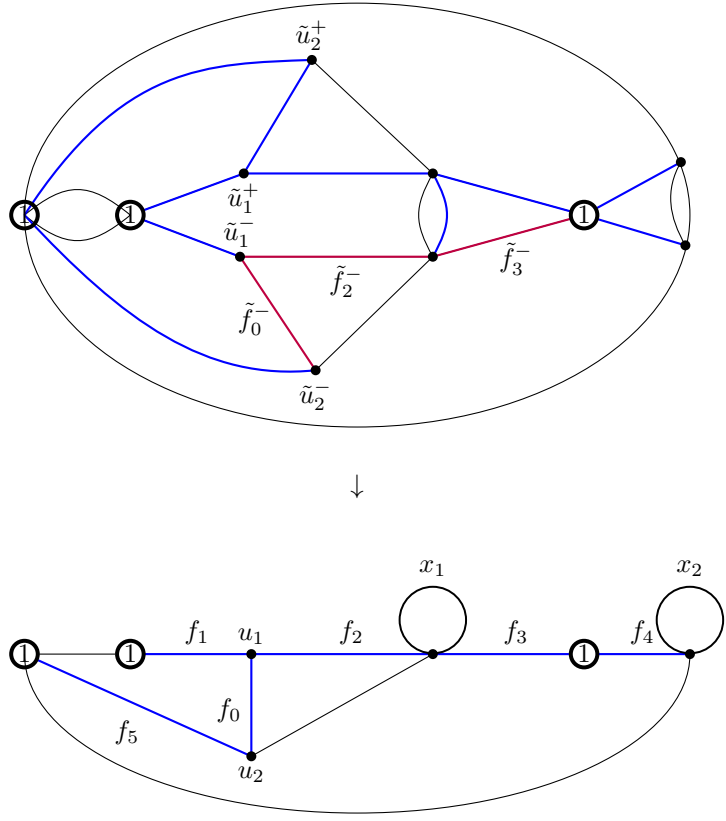
\begin{figure}[h]
    \centering
\begin{tikzpicture}

% radius (in em) used for shortening at circled vertices
\def\Rc{0.55em}

% ============================================================
%  TOP PANEL
% ============================================================
\begin{scope}[shift={(0,5.8)}]

\coordinate (L0) at (-4.4, 0.0);
\coordinate (L1) at (-3.0, 0.0);

\coordinate (u1p) at (-1.5, 0.55);
\coordinate (u2p) at (-0.6, 2.05);
\coordinate (T)   at ( 1.0, 0.55);
\coordinate (B)   at ( 1.0,-0.55);

\coordinate (u1m) at (-1.55,-0.55);
\coordinate (u2m) at (-0.55,-2.05);

\coordinate (R1) at ( 3.0, 0.00);
\coordinate (R2) at ( 4.28, 0.70);
\coordinate (R3) at ( 4.34,-0.40);

% outer boundary
\draw[thin, black] (0,0) ellipse (4.4 and 2.8);

% small double edge between L0 and L1
\draw[thin, black]
  (L0) .. controls (-3.85, 0.45) and (-3.55, 0.45) .. (L1)
       .. controls (-3.55,-0.45) and (-3.85,-0.45) .. (L0);

% internal thin edges (shortened to hit outer border of circled vertices)
\draw[thin, black, shorten <=\Rc] (L1) -- (u1p);
\draw[thin, black, shorten <=\Rc] (L1) -- (u1m);

\draw[thick, blue] (u1p) -- (u2p);
\draw[thin, black] (u2p) -- (T);

\draw[thin, black] (u1m) -- (B);
%\draw[thin, black] (B)   -- (T);

\draw[thin, black, shorten >=\Rc] (B) -- (R1);
\draw[thin, black, shorten >=\Rc] (T) -- (R1);
%\draw[thin, black, shorten <=\Rc] (R1) -- (R2);
%\draw[thin, black, shorten <=\Rc] (R1) -- (R3);

% lens: two thick curves + arrow on middle curve
\draw[thin, black]
  (B) .. controls (0.75,-0.10) and (0.75, 0.10) .. (T);
\draw[thick, blue]
  (B) .. controls (1.25,-0.10) and (1.25, 0.10) .. (T);

\draw[thin, black] (R2) to [out=180+60, in=90+30] (R3);

%\draw[thin, black,  postaction={decorate}, decoration={markings, mark=at position 0.55 with {\arrow{stealth}}}] (B) .. controls (1.00,-0.10) and (1.00, 0.10) .. (T);

% highlighted family 1 (blue)
\draw[thick, blue]
  (L0) .. controls (-3.1, 2.0) and (-2.0, 2.0) .. (u2p);
\draw[thick, blue]
  (L0) .. controls (-3.1,-1.4) and (-2.0,-2.2) .. (u2m);

\draw[thick, blue, shorten <=\Rc] (L1) -- (u1p);
\draw[thick, blue] (u1p) -- (T);
\draw[thick, blue, shorten >=\Rc] (T) -- (R1);
\draw[thick, blue, shorten <=\Rc] (R1) -- (R2);
\draw[thick, blue, shorten <=\Rc] (R1) -- (R3);
\draw[thick, blue, shorten <=\Rc] (L1) -- (u1m);

% highlighted family 2 (purple)
\draw[thick, purple] (u1m) -- (B);
\draw[thick, purple, shorten >=\Rc] (B) -- (R1);
\draw[thick, purple] (u1m) -- (u2m);
\draw[thin, black] (u2m) -- (B);
% filled nodes
\foreach \p in {u1p,u2p,T,B,u1m,u2m,R2,R3}
  \draw[fill=black] (\p) circle (0.15em);

% circled vertices
\draw[line width=1.5pt,black] (L0) circle (0.50em);
\draw[line width=1.5pt,black] (L1) circle (0.50em);
\draw[line width=1.5pt,black] (R1) circle (0.50em);

% numbers "1" inside empty circles
\node at (L0) {$1$};
\node at (L1) {$1$};
\node at (R1) {$1$};

% labels
\node[above]      at (u2p) {$\tilde u_2^{+}$};
\node[below ] at (u1p) {$\tilde u_1^{+}$};
\node[above ] at (u1m) {$\tilde u_1^{-}$};
\node[below]      at (u2m) {$\tilde u_2^{-}$};

\node[left]  at ($(u1m)!0.55!(u2m)$) {$\tilde f_0^-$};
\node[below]       at ($(u1m)!0.55!(B)$)   {$\tilde f_2^-$};
\node[below] at ($(B)!0.55!(R1)$)    {$\tilde f_3^-$};

\end{scope}

% arrow down between panels
\draw (0,2.5) node [below] {$\downarrow$};

% ============================================================
%  BOTTOM PANEL
% ============================================================
\begin{scope}[shift={(0,0)}]

\coordinate (A)  at (-4.4,0.0);
\coordinate (A2) at (-3.0,0.0);
\coordinate (U1) at (-1.4,0.0);
\coordinate (C)  at ( 1.0,0.0);
\coordinate (R)  at ( 3.0,0.0);
\coordinate (E)  at ( 4.4,0.0);

\coordinate (U2) at (-1.4,-1.35);

% outer boundary: lower half ellipse
\draw[thin, black] (E) arc[start angle=0, end angle=-180, x radius=4.4, y radius=2.1];

% baseline: split into 3 pieces so it does NOT pass through circled vertices (A2 and R)
%\draw[thin, black, shorten <=\Rc] (A) -- (A2);
%\draw[thin, black, shorten <=\Rc, shorten >=\Rc] (A2) -- (R);
%\draw[thin, black, shorten >=\Rc] (R) -- (E);

\draw[thin, black, shorten <=\Rc] (A) -- (-3.2,0);
\draw[thick, blue, shorten <=\Rc] (-3,0) -- (2.8,0);
\draw[thick, blue, shorten <=\Rc] (R) -- (E);

% left triangle

\draw[thick, blue] (U1) -- (U2);
\draw[thin, black] (C) -- (U2);
\draw[thick, blue, shorten <=\Rc] (A) -- (U2);
%\draw[thin, black] (U2) -- (C);

% highlighted baseline segments (already shortened where needed)
%\draw[thick, blue, shorten >=\Rc] (A2) -- (U1);
%\draw[thick, blue]               (U1) -- (C);
%\draw[thick, blue, shorten <=\Rc] (C)  -- (R);

% add the missing side to close the triangle containing e_0^{(2)}
%\draw[thin, black] (U1) -- (C);

% loop x1 at C: close to a circle
\coordinate (Ptop) at ($(C)+(0.0,0.45)$);
;
\node[above] at ($(C)+(0.0,0.9)$) {$x_1$};
\draw[black] (Ptop)  circle (1.25em);
\draw[black] (4.4,0.45)  circle (1.25em);
\node[above] at ($(E)+(0.0,0.9)$) {$x_2$};
% nodes
\foreach \p in {U1,C,U2,E}
  \draw[fill=black] (\p) circle (0.15em);

% circled vertices
\draw[line width=1.5pt,black] (A)  circle (0.50em);
\draw[line width=1.5pt,black] (A2) circle (0.50em);
\draw[line width=1.5pt,black] (R)  circle (0.50em);

% numbers "1" inside empty circles
\node at (A) {$1$};
\node at (A2) {$1$};
\node at (R) {$1$};

% labels
\node[above] at (U1) {$u_1$};
\node[below] at (U2) {$u_2$};

\node[above] at ($(A2)!0.55!(U1)$) {$f_1$};

\node[above] at ($(R)!0.55!(E)$) {$f_4$};

\node[left]  at ($(U1)!0.55!(U2)$) {$f_0$};
\node[above] at ($(U1)!0.55!(C)$)  {$f_2$};
\node[above] at ($(C)!0.55!(R)$)   {$f_3$};
\node[below left] at ($(A)!0.55!(U2)$) {$f_5$};

\end{scope}

\end{tikzpicture}
 \caption{A harmonic double cover with $n_d=3$ and $k=1$.}
    \label{ex:dilatedpaths}
\end{figure}

\subsection{The set \texorpdfstring{$\mathcal{B}_{\wt T}$}{B{T}} generates the lattice \texorpdfstring{$\Im (\Id-\iota_*)$}{Im(Id-iota*)}}

Before proving that the previous set of independent anti-symmetric cycles is a basis for the lattice $\Im (\Id-\iota_*)$, let us illustrate with an example how to write any antisymmetric cycle in terms of the elements in $\mathcal{B}_{\wt T}$.

\begin{ex} Consider the harmonic double cover with two isolated dilated vertices in
\cref{ex:graphwithbasis}. Choose the spanning tree $T=\{f_2,f_1,f_0,f_3\}$ highlighted in blue and construct the admissible spanning tree $\wt T$ where $\tilde f_0^-\not \in\wt T$.
Consider an orientation on $\Ga$ and the induced orientation on $\tGa$. 
Since the cycles $\beta_e\in\mathcal B_{\wt T}$ are of the form $\tilde \ga_{e_i}^--\iota\tilde \ga_{e_i}^-$, let us prove that the remaining anti-symmetric cycles, namely $\tilde \ga_{e_i}^+-\iota\tilde \ga_{e_i}^+$, can be obtained as $\ZZ$-linear combination of those.
Consider first the edge $\tilde e_2^+\in E(\tGa\setminus\wt T)$. The cycle $\tilde \ga_{e_2}^+$ is contained in $\pi^{-1}(T)^+$ and it is $\tilde \ga_{e_2}^+=\tilde e_2^+-\tilde f_1^+$. Note that, in this case, $\iota\tilde \ga_{e_2}^+=\tilde e_2^--\tilde f_1^-=\tilde \ga_{e_2}^-$. Therefore, the anti-symmetric cycle 
\[\tilde \ga_{e_2}^+-\iota\tilde \ga_{e_2}^+=-(\tilde \ga_{e_2}^--\iota\tilde \ga_{e_2}^-)=-\beta_{e_2},\] which is an element in the span of $\mathcal B_{\wt T}$, since $\beta_{e_2}\in \mathcal B_{\wt T}$. 
A similar situation happens when considering the cycle $\tilde \ga_{e_4}^+$. It corresponds to $\tilde \ga_{e_4}^+=\tilde e_4^++\tilde f_3^--\tilde f_3^+$ and contains both the preimages of $f_3\in T$. Therefore,  $\iota\tilde \ga_{e_4}^+=\tilde e_4^-+\tilde f_3^+-\tilde f_3^-$
which in turn is equal to $\tilde\ga_{e_4}^-=\tilde e_4^-+\tilde f_3^+-\tilde f_3^-$. Thus, 

\[\tilde \ga_{e_4}^+-\iota\tilde \ga_{e_4}^+=-(\tilde \ga_{e_4}^--\iota\tilde \ga_{e_4}^-)=-\beta_{e_4}.\]

Consider now the edge $\tilde e_1^+$.
The cycle $\tilde\ga_{e_1}^+$ corresponds to
$
\tilde\ga_{e_1}^+=\tilde e_1^++\tilde f_3^--\tilde f_3^+-\tilde f_0^+-\tilde f_1^+,$
while the cycle $\tilde\ga_{e_1}^-=\tilde e_1^--\tilde f_0^+-\tilde f_1^+-\tilde f_2^++\tilde f_2^-$.
Therefore,

\begin{equation*}
    \begin{split}
        \tilde\ga_{e_1}^+-\iota \tilde\ga_{e_1}^+&=\tilde e_1^++\tilde f_3^--\tilde f_3^+-\tilde f_0^+-\tilde f_1^+-\tilde e_1^--\tilde f_3^++\tilde f_3^-+\tilde f_0^-+\tilde f_1^-,
    \end{split}
\end{equation*}

while

\[
\tilde\ga_{e_1}^--\iota\tilde\ga_{e_1}^-=\tilde e_1^--\tilde f_0^+-\tilde f_1^+-\tilde f_2^++\tilde f_2^--\tilde e_1^++\tilde f_0^-+\tilde f_1^-+\tilde f_2^--\tilde f_2^+,
\]

and the anti-symmetric cycle associated to the edge $\tilde f_0^- $ is 
\[\beta_{f_0}=2\tilde\ga_{f_0}=2\tilde f_0^-+2\tilde f_3^--2\tilde f_3^+-2\tilde f_0^+-2\tilde f_1^+-2\tilde f_2^++2\tilde f_2^-+2\tilde f_1^-.\] 

By putting everything together, the cycle $\tilde\ga_{e_1}^+-\iota \tilde\ga_{e_1}^+$ can be written as

\[\tilde\ga_{e_1}^+-\iota \tilde\ga_{e_1}^+=\beta_{f_0}-(\tilde \ga_{e_1}^--\iota\tilde \ga_{e_1}^-) =\beta_{f_0}-\beta_{e_1},\] and again both $\beta_{f_0}$ and $\beta_{e_1}$ belong to $\mathcal B_{\wt T}$.
Finally, let us consider the edge $\tilde e_3^+.$ The cycle $\tga_{e_3}^+=\tilde e_3^+-\tilde f_0^+$ and $\iota \tga_{e_3}^+=\tilde e_3^--\tilde f_0^-$. Moreover,  
\[\tga_{e_3}^-=\tilde e_3^-+\tilde f_3^--\tilde f_3^+-\tilde f_0^+-\tilde f_1^+-\tilde f_2^++\tilde f_2^-+\tilde f_1^-.\] 
Therefore, the anti-symmetric cycle is
\[\tga_{e_3}^--\iota\tga_{e_3}^-=\tilde e_3^--\tilde e_3^++2\tilde f_3^--2\tilde f_3^+-\tilde f_0^++\tilde f_0^--2\tilde f_1^+-2\tilde f_2^++2\tilde f_2^-+2\tilde f_1^-,\]

and
\[
\beta_{f_0}-(\tga_{e_3}^--\iota\tga_{e_3}^-)=   \tilde e_3^+-\tilde e_3^-+\tilde f_0^--\tilde f_0^+=\tga_{e_3}^+-\iota\tga_{e_3}^+.
\]

We have written every anti-symmetric cycle as linear combination of cycles  $\beta_e\in\mathcal B_{\wt T}$.

Note that, if instead of $\wt T$ we choose $\wt T'=\wt T\setminus\{\tilde f_3\}\cup \tilde e^+_4$, then $\tilde \ga_{e_1}^+$ does not pass through the edges $\tilde f_3^{\pm}$ anymore, same for $\beta_{f_0}$. Then,

\begin{equation*}
    \begin{split}
        \tilde \ga_{e_1}^+-\iota\tilde \ga_{e_1}^+&=\tilde e_1^+-\tilde e_4^+-\tilde f_0^+-\tilde f_1^+-\tilde e_1^-+\tilde e_4^-+\tilde f_0^-+\tilde f_1^-,\\
\beta_{f_0}&=2\tilde f_0^--\tilde e_4^++\tilde e_4^--2\tilde f_0^+-2\tilde f_1^+-2\tilde f_2^++2\tilde f_2^-+2\tilde f_1^-,\\
\beta_{e_1}&=\tilde e_1^--\tilde f_0^+-\tilde f_1^+-\tilde f_2^++\tilde f_2^--\tilde e_1^++\tilde f_0^-+\tilde f_1^-+\tilde f_2^--\tilde f_2^+
    \end{split}
\end{equation*}

Nonetheless, we obtain again 
$
\beta_{f_0}-\beta_{e_1}= \tilde \ga_{e_1}^+-\iota\tilde \ga_{e_1}^+.
$
Similarly for the cycle $\tilde \ga_{e_3}^+$.
\begin{figure}[H]
    \centering
\begin{tikzpicture}
%primo vertice a sx
%\coordinate (-1) at (-2,0);
%\coordinate (-1a) at (-2,2);
%\coordinate (-1b) at (-2,3);
\coordinate (0) at (-1,0);
\coordinate (0a) at (-1,2.5+0.5);

\coordinate (1) at (0,0); %leftvertex
 \coordinate (a1) at (0,2); %preimage - of the vertex (1)
\coordinate (b1) at (0,4); %preimage + of the vertex (1)
\coordinate (2) at (1+1,0); %rightvertex
 \coordinate (a2) at (1+1,2); %preimage - of the vertex (2)
 \coordinate (b2) at (1+1,4); %preimage + of the vertex (2)
 
 \coordinate (12) at (1,0);
 \coordinate (12a) at (1,2);
 \coordinate (12b) at (1,4);

 \coordinate (3) at (2+1,0);
 \coordinate (a3) at (2+1,2.5+0.5);

 \coordinate (4) at (2.2+1,0);
  \coordinate (a4) at (2.2+1,2.5+0.5);
  
 \coordinate (c1) at (1.4-0.5,2.4+0.5); 
 \coordinate (c2) at (1.6-0.5,2.6+0.5);
 %contruction of the edge in the middle: c1 and c2 are the middle points so that the segment is cutted in half

 \foreach \i in {2,12,12a,12b,1,a2,a1,b2,b1}
 \draw[fill=black](\i) circle (0.15em);

\draw[line width=1.5pt,black](0) circle (0.50em);
\draw[line width=1.5pt,black](0a) circle (0.50em);
\draw[line width=1.5pt,black](4) circle (0.50em);
\draw[line width=1.5pt,black](a4) circle (0.50em);
%\draw[thin,black](-1) circle (0.50em);
%\draw[thin,black](-1a) circle (0.50em);
%\draw[thin,black](-1b) circle (0.50em);

\draw[thin, black] (-0.8,0)--(1);
\draw[thin, black] (2)--(1);
\draw[thin, black] (a2)--(a1);
\draw[thick, blue] (b2)--(b1);
%\draw[thin, black] (-1.2,0)--(-1.8,0);
 \draw[thin, black] (1) to [out=90, in=90] (12);
  \draw[thin, black] (12) to [out=90, in=90] (2);
\draw[thin, black] (1) to [out=90+180, in=-90] (2);

%loop

 \draw[thin, black] (2) to [out=-90, in=180+45] (2.4,-0.5);
  \draw[thin, black] (2) to [out=0, in=45] (2.4,-0.5);

  %preimm loop
  \draw[thin, black] (b2) to [out=-60, in=60] (a2);
  \draw[thin, black] (b2) to [out=180+60, in=90+30] (a2);

\draw[thin, black] (b1) to [out=90, in=90] (12b);
\draw[thin, black] (12b) to [out=90, in=90] (b2);
\draw[thin, black] (12a) to [out=90+180, in=-90] (a2);
\draw[thin, black] (a1) to [out=90+180, in=-90] (12a);

\draw[thick, blue] (3,0)--(-0.8,0);

\draw[thick, blue] (a1)--(-0.8,3);
\draw[thick, blue] (a1)--(12a);
\draw[thick, blue] (b1)--(-0.8,3);
  \draw[thin, black] (b1)--(a2);
\draw[thin, black] (a1)--(c1);
    \draw[thick, blue] (a3)--(a2);
    \draw[thick, blue] (a3)--(b2);
    \draw[thick, blue] (2)--(3);
    \draw[thin, black] (b2)--(c2);

  \draw (1,1.3) node [below] {$\downarrow$};
    \draw (0) node  {$1$};
     \draw (0a) node  {$1$};
  \draw (4) node  {$1$};
   \draw (a4) node  {$1$};
    \draw (1,-0.8) node  {$e_1$};
     \draw (-0.5,-0.2) node  {$f_2$};
     \draw (2.7,-0.2) node  {$f_3$};
      \draw (2.5,-0.7) node  {$e_4$};
    \draw (1.5,-0.2) node  {$f_0$};
    \draw (1.5,0.5) node  {$e_3$};
    \draw (0.5,-0.2) node  {$f_1$};
    \draw (0.5,0.5) node  {$e_2$};
      \draw (0.45,3.2) node  {$\tilde e_1^+$};
        \draw (0.5,2.8) node  {$\tilde e_1^-$};

         \draw (-0.5,3.7) node  {$\tilde f_2^+$};
        \draw (-0.5,2.3) node  {$\tilde f_2^-$};
        \draw (2.7,3.7) node  {$\tilde f_3^+$};
        \draw (2.7,2.3) node  {$\tilde f_3^-$};
 \draw (2,3.5) node  {$\tilde e_4^+$};
        \draw (2,2.5) node  {$\tilde e_4^-$};
        
    \draw (1.4,3.8) node  {$\tilde f_0^+$};
    \draw (1.4,3.8+0.8) node  {$\tilde e_3^+$};
     \draw (1.4,2.2) node  {$\tilde f_0^-$};
     \draw (1.4,2.2-0.75) node  {$\tilde e_3^-$};
    \draw (0.5,3.8) node  {$\tilde f_1^+$};
     \draw (0.5,3.8+0.8) node  {$\tilde e_2^+$};
     \draw (0.5,2.2) node  {$\tilde f_1^-$};
 \draw (0.5,2.2-0.75) node  {$\tilde e_2^-$};
  \end{tikzpicture}
 \caption{A harmonic double cover and an admissible spanning  tree $\wt T$ highlighted in blue.}
    \label{ex:graphwithbasis}
\end{figure}
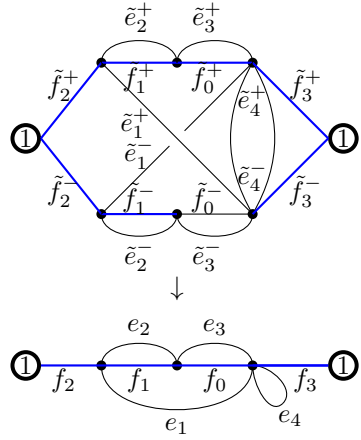
\end{ex}

Since we are dealing with sets of oriented edges, let us introduce the following notation.
\begin{defn}\label{orientedpath}
    Let $G$ be a finite graph endowed with an orientation and let $T$ be a spanning tree of the graph. Then, any edge $e_i=u_iv_i\in E(G\setminus T)$ with $u_i,v_i\in V(G)$ determines a unique path in $T$ between the endpoints $u_i$ and $v_i$.  Construct the cycle $\ga_{e_i}=T\cup e_i$ that satisfies $\langle \ga_{e_i},e_i\rangle=1$.
    Define $\underrightarrow{P_i}$ to be the \emph{oriented path} in $T$ such that $\ga_{e_i}=\underrightarrow{P_i}+e_i$. And denote by $P_i$ the set of edges appearing in the oriented path $\underrightarrow{P_i}$ without orientation. Analogously, denote by $\underrightarrow{\wt P_i^{+}}$ and $ \underrightarrow{\iota\wt P_i^{+}}$ the two preimages of the path $\underrightarrow{P_i}$ in $\tGa$.

 Furthermore, from \cref{rmk:edgeofthesecondkind}, for each  edge $f_i\in E^{\wt T}_{\pi}$, let $\underrightarrow{Q_i}$
  be the unique oriented path in $T$    with endpoints in $V(\Ga_{\dil})\cup V(P^{\wt T}_{\pi})$ that contains $f_i$. Define
 $\underrightarrow{\wt Q_i^+}:=\pi^{-1}(\underrightarrow{Q_i})^+$ and $\underrightarrow{\iota \wt Q_i^+}=\pi^{-1}(\underrightarrow{Q_i})^-$   to be the unique oriented paths in $\wt \Gamma$ that contain respectively the edges $\tilde f_i^{\pm}$ such that  $f_i\in E^{\wt T}_{\pi}$. Recall that the anti-symmetric cycle $\beta_{f_i}=\tga_{f_i}^--\iota\tga_{f_i}^-$ intersects positively the edge $\tilde f_i^-$, which belongs to $\iota \wt Q_i^+$.
   With this notation, the cycle $\beta_{f_i}$ can be written as follows 
   \begin{equation}
   \beta_{f_i}=
       \begin{cases}
           2\underrightarrow{\iota \wt Q_i^+}-2\underrightarrow{ \wt Q_i^+} & \text{if }  Q_i\cap V(P^{\wt T}_{\pi})=\emptyset,\\
            2\underrightarrow{\iota \wt Q_i^+}-2\underrightarrow{ \wt Q_i^+}+\tilde x_1^+-\tilde x_1^- & \text{if }Q_i\cap V(P^{\wt T}_{\pi})=\{v_1\},\\
            2\underrightarrow{\iota \wt Q_i^+}-2\underrightarrow{ \wt Q_i^+}+\tilde x_1^+-\tilde x_1^-+\tilde x_2^+-\tilde x_2^- &\text{if } Q_i\cap V(P^{\wt T}_{\pi})=\{v_1,v_2\}.
       \end{cases}
   \end{equation}
   where $v_1,v_2$ are the vertices whose loop attached is $x_1,x_2$. 
    %\MM{ma ti servono questi cammini solo a partire da lati di $G\setminus T$ del grafo di sotto?}\giusi{No, ho aggiunto anche  sui cammini $Q_i$ all'interno dell'albero. }
\end{defn}

We are now ready to prove that the set $\mathcal B_{\wt T}$ is a basis for the lattice  $ \Im (\Id-\iota_*)$.

\begin{rmk}\label{rmk:basis}

Let $T$ be a spanning tree for $\Ga$ and $\wt T$ be an admissible   spanning tree for    $\tGa$.
   We have already seen that, for any choice of $\wt T$, the cycles $\beta_e\in\mathcal B_{\wt T}$ are all independent in Lemma \ref{lem:independent(a)} and Lemma \ref{lem:independent(b)},
  therefore,
  in order to prove that the set $\mathcal B_{\wt T}$ is a basis for the lattice  $ \Im (\Id-\iota_*)$,  
  we need to show  that any cycle in the lattice  $ \Im (\Id-\iota_*)$ is a $\ZZ$-linear combination of the elements in $\mathcal{B}_{\wt T}$ for any  choice of admissible tree $\wt T$. By definition, any cycle in  $ \Im (\Id-\iota_*)$ comes from a cycle in $H_1(\tGa,\ZZ)$ under the homomorphism \[\Id-\iota_*\colon H_1(\tGa,\ZZ)\to H_1(\tGa,\ZZ).\]

Assign an orientation on $\Ga$ and consider the induced orientation on $\tGa$. 
 Let
    $\mathcal B=\{\tga_i\}_{i=1,\dots,b_1(\tGa)}$ be a basis for $H_1(\tGa,\ZZ)$, where each $\tga_i$ corresponds to the smallest  cycle in $\wt T\cup \{\tilde e_i\}$, with $\tilde e_i\in E(\wt\Ga\setminus \wt T)$.
    Now let $\alpha\in H_1(\tGa,\ZZ)$ be any cycle. Then, $\alpha=\sum_{i=1}^{b_1(\tGa)}a_i\tilde\gamma_{i}$, with $a_i\in\ZZ$.  
    Applying the homomorphism $\Id-\iota_*$ to the cycle $\alpha$, we get
    \[\Id-\iota_*(\alpha)=\sum_{i=1}^{b_1(\tGa)}a_i(\tilde\gamma_{i}-\iota \tilde\gamma_{i}).
    \]

 Since $\tGa=\pi^{-1}(\Ga)$, any %cycle $\tga_i\in H_1(\tGa,\ZZ)$ intersects some edge $\tilde e_i^+$ or $\tilde e_i^-$, in the complement of the spanning tree $\wt T$, with $e_i\in E(\Ga\setminus T)$,
  edge $\tilde e_i\in E(\tGa\setminus \wt T) $ is of the form $\tilde e_i^+$ or $\tilde e_i^-$ where $\pi(\tilde e_i^{\pm})=e_i\in E(\Ga\setminus T)$. Therefore, the cycle $\tga_i$
corresponds to either $\tga_{e_i}^+$ or $\tga_{e_i}^-$ as in \cref{construction}.
 Moreover, since each $\beta_e\in \mathcal{B}_{\wt T}$ is such that $\beta_e=\tilde\gamma_e^{-}-\iota \tilde\gamma_e^{-}$ with $e\in E_{nz}$, it suffices to show that any cycle $\tilde\gamma_{e_i}^{+}-\iota \tilde\gamma_{e_i}^{+}$, with $e_i\in E(\Ga\setminus T)$, can be expressed as linear combination of such anti-symmetric cycles $\beta_e$ with $e\in E_{nz}$.
   Furthermore, without loss of generality, we may assume that the dilation subgraph consists only of isolated dilated vertices, and we denote by $n_d$ their total number. Indeed, if 
a cycle $\tilde\ga_{e_i}^+$ intersects some dilated edge $\tilde d_i$, where $\tilde d_i=\tilde d_i^+=\tilde d_i^-$, then $\iota\tilde \ga_{e_i}^+$ also contains $\iota(\tilde d_i)=\tilde d_i$. In the difference $\tilde\ga_{e_i}^+-\iota\tilde \ga_{e_i}^+$, the two edges appear with opposite orientation, yielding $\tilde d_i-\tilde d_i=0$. Consequently,  the anti-symmetric cycle
$\tilde\ga_{e_i}^+-\iota\tilde \ga_{e_i}^+$
contains only non-dilated edges, and it is sufficient to prove that those edges
can be obtained by a $\ZZ$-linear combination of the elements of $\mathcal B_{\wt T}$. 
\end{rmk}

\begin{thm}\label{thm:basis}
     Let $\pi\colon\tGa\to\Ga$ and $\wt T$ and $T$ be as above. Then $\mathcal{B}_{\wt T}$ is a basis for the lattice $\Im (\Id-\iota_*)$ for any choice of admissible tree $\wt T$.
\end{thm}
\begin{proof}
 By \cref{rmk:basis}, it is sufficient to prove that  any anti-symmetric cycle $\tilde\gamma_{e_i}^{+}-\iota \tilde\gamma_{e_i}^{+}$, with $e_i\in E(\Ga\setminus T)$, can be expressed as linear combination of $\beta_e$ with $e\in E_{nz}$.  
Let us distinguish \cref{1.a}.(a) from \cref{1.b}.(b):

\begin{itemize}
    \item[(1.a)] 
The double cover $\pi$ is free with $P^{\wt T}_{\pi}=\emptyset$ as in \cref{construction}.(a),

\item[(1.b)] Either the double cover $\pi$ is free with $P_{\pi}^{\wt T}\ne\emptyset$ or $\pi$ is dilated with $n_d\geq 1$  and  $|P^{\wt T}_{\pi}|=k\geq 0$,  as in \cref{construction}.(b).

\end{itemize}

In case $(1.a)$, the admissible tree is $\wt T=\pi^{-1}(T)\cup \pi^{-1}(T)^-\cup \{\tilde e_g^+\}$, where $e_g\in E(\Ga\setminus T)$. Any cycle $\tilde\ga_{e_i}^+$ satisfies $\iota\tilde\ga_{e_i}^+=\tilde\ga_{e_i}^-$; thus 
$\tilde\ga_{e_i}^+-\iota\tilde\ga_{e_i}^+=-(\tilde\ga_{e_i}^--\iota\tilde\ga_{e_i}^-)$
and $\mathcal{B}_{\wt T}$ is a basis for the lattice.

\medskip

In case $(1.b)$, we will prove that, for each $e_i\in E(\Ga\setminus T)$, the linear combination of the anti-symmetric cycle $\tga_{e_i}^+-\iota\tga_{e_i}^+$,  is the following:
\begin{equation}\label{eq:generalequationombination}
      \tga_{e_i}^+-\iota\tga_{e_i}^+= \sum_{f_j\in E^{\wt T}_{\pi}}-\langle\tga_{e_i}^+,\tilde f_j^+\rangle\cdot \beta_{f_j}-\beta_{e_i}.
 \end{equation}

Let $e_i=u_iv_i\in E(\Ga\setminus T)$ and let $\underrightarrow{P_i}$, $\underrightarrow{\wt P_i^{+}}$ and $\underrightarrow{\iota\wt P_i^+}$  be the oriented paths associated to $\ga_{e_i}$ and $\tga_{ e_i}^+$, as in \cref{orientedpath}.
Therefore, the cycle $ \tga_{e_i}^+ $ can be written as

\begin{equation}\label{cicloei}
\begin{cases}
   \tga_{e_i}^+=\tilde e_i^++\underrightarrow{\wt P_i^+}+\tilde x_i^+ & \text{if} \quad P_i\cap V(P^{\wt T}_{\pi})\ne\emptyset\quad \text{and}\quad\tilde e_i^+=u_i^+v_i^-, \\
   \tga_{e_i}^+=\tilde e_i^++\underrightarrow{\wt P_i^+} & \,\,\,\,\quad \qquad \qquad\qquad\qquad\quad\qquad \text{otherwise}
\end{cases}
\end{equation}

depending on whether the path $P_i$ in $T$ contains vertices of $V(P^{\wt T}_{\pi})$ or not, and we denote by $x_i$ the $\pi$-parallel loops. 
Thus, the anti-symmetric cycle $\tga_{e_i}^+-\iota \tga_{e_i}^+$ is 
\[
\begin{cases}

\tga_{e_i}^+-\iota \tga_{e_i}^+ =\tilde e_i^++\underrightarrow{\wt P_i^+} -  \tilde e_i^--\underrightarrow{\iota\wt P_i^+}+\tilde x_i^+-\tilde x_i^- & \text{if} \quad P_i\cap V(P^{\wt T}_{\pi})\ne\emptyset\quad \text{and}\quad\tilde e_i^+=u_i^+v_i^-,\\

\tga_{e_i}^+-\iota \tga_{e_i}^+ =\tilde e_i^++\underrightarrow{\wt P_i^+} -  \tilde e_i^--\underrightarrow{\iota\wt P_i^+} & \,\,\,\,\quad \qquad \qquad\qquad\qquad\quad\qquad \text{otherwise}.
\end{cases}
\]
By \cref{eq1}, there are $n_d-1+k$ cycles $\beta_{f_i}$ such that $\tilde f_i^+\in\wt T$ but $\tilde f_i^-\not \in \wt T$, for  $i=1,\dots,n_d-1+k$.
Moreover, by \cref{rmk:edgeofthesecondkind}, each edge $f_i\in E^{\wt T}_{\pi}$  
defines a unique path $Q_i\in T$  
where we denote by $\underrightarrow{\wt Q_i^+}$ and $\underrightarrow{\iota\wt Q_i^+}$ the oriented paths in $\wt \Gamma$ for the cycle $\beta_{f_i}$ as in \cref{orientedpath}.

If $\tilde \ga_{e_i}^+$ does not intersect any of  the edges $\tilde f_i^+$, then we have 
$\iota\tilde \ga_{e_i}^+=\tga_{e_i}^-$ and the result follows
as in the previous case. From now on, we will assume that  $\tilde \ga_{e_i}^+$  passes through some of  the edges $\tilde f_i^+$, for  $1\leq i\leq n_d-1+k$.
First, suppose  there is exactly one edge $ f_1\in E^{\wt T}_{\pi}$. Assume, moreover, that $\langle\tga_{e_i}^+,\tilde f_1^+\rangle=-1$, the other case is analogous and will be discussed in the more general setting later. 
Then, if $|E^{\wt T}_{\pi}|=1$, we have the following cases: $(n_d,k)=(0,2)$, $(1,1)$, or $(2,0)$.

\begin{itemize}
    \item  The pair $(n_d,k)=(0,2)$: in this case there are $2=k=|P^{\wt T}_{\pi}|$ $\pi$-parallel loops, and denote them by $x_1,x_2$.  
Note that the paths $\underrightarrow{\wt Q_1^+},\underrightarrow{\wt P_i^+}$ 
belong to $\wt T$ and pass through the edge $\tilde f_1^+$, while the paths $\underrightarrow{\iota\wt Q_1^+},\underrightarrow{\iota \wt P_i^+}$ do not, because $\tilde f_1^-\not\in \wt T$.
Let us compute the cycle $\tga_{e_i}^-$.  Since it cannot contain $\tilde f_1^-\in \iota\wt  P_i^+$ and it has to be contained in $\wt T\cup \{\tilde e_i^-\}\subset \pi^{-1}(T)\cup \pi^{-1}(P_{\pi})^{+}\cup \{\tilde e_i^-\}$, then it will pass through the edges in $\iota\wt P_i^+\Delta\iota \wt Q_1^+$ and $\wt Q_1^+$ and in $\tilde x_1^+,\tilde x_2^+$ depending on whether the edge $\tilde e_i^-=\tilde u_i^-\tilde v_i^-$ or  $\tilde e_i^-=\tilde u_i^-\tilde v_i^+$ (see for instance the edges $\tilde e_1^-$ and $\tilde e_2^-$ in \cref{ex:condition+-}). 
 Here, denote by  $\underrightarrow{\iota\wt P_i^+\Delta\iota \wt Q_1^+}$  the path obtained from the edges in 
 $\iota\wt P_i^+\Delta\iota \wt Q_1^+$
 with orientation prescribed by $\underrightarrow{\iota\wt P_i^+}$ and $\underrightarrow{\iota \wt Q_1^+}$. %\giusi{The thing is that, in the intersection the edges have different orientation, therefore taking the operation $\Delta$ of those sets will have the edges twice with opposite orientation and sum up to zero. But I wanted to stress the fact that they cannot pass through those edges, therefore i like the $\Delta$ operation.}
Therefore,

\[
\begin{cases}
    \tga_{e_i}^-=\tilde e_i^-+
%\sum_{e\in \iota\wt  Q_1^+\Delta \iota\wt P_i^+}
\underrightarrow{\iota\wt P_i^+\Delta\iota \wt Q_1^+}
 -\underrightarrow{ \wt Q_1^+} + \tilde x_1^++ \tilde x_2^+& \text{if} \quad \tilde e_i^-=u_i^-v_i^-,\\
   \tga_{e_i}^-=\tilde e_i^-+
%+\sum_{e\in \iota\wt  Q_1^+\Delta \iota\wt P_i^+}e
\underrightarrow{\iota\wt P_i^+\Delta\iota \wt Q_1^+}
-\underrightarrow{ \wt Q_1^+} + \tilde x_1^+ & \text{if} \quad \tilde e_i^-=u_i^-v_i^+,
\end{cases}
\] 

 where we assume the orientation on the preimages $\pi$-parallel loops to be positive (the other case is analogous). 

 The difference between the two situations comes from the fact that if the cycle is contained in $\pi^{-1}(T)^+$, then $\tga_{e_i}^-$ is not contained in $\pi^{-1}(T)^-$ and it must traverse the two trees in correspondence of the (unique) two preimages of the $\pi$-parallel loops in order to form a simple cycle, otherwise it passes through only one $x_i$, WLOG $x_1$. Therefore, the cycle  $\iota\tga_{e_i}^-$ will be of the form:

%\[\begin{cases}    \iota \tga_{e_i}^-=\tilde e_i^++\sum_{e\in \wt Q^+\Delta \wt P_i^+ }e -\underrightarrow{\iota\wt Q_1^+} + \tilde x_1^-+ \tilde x_2^-& \text{if} \quad \tilde e_i^-=\tilde u_i^-\tilde v_i^-,\\  \iota \tga_{e_i}^-=\tilde e_i^++\sum_{e\in \wt Q^+\Delta \wt P_i^+ }e-\underrightarrow{\iota\wt Q_1^+}  + \tilde x_1^- & \text{if} \quad \tilde e_i^-=\tilde u_i^-\tilde v_i^+.\end{cases}\]
\[
\begin{cases}
    \iota\tga_{e_i}^-=\tilde e_i^++
%\sum_{e\in \wt  Q_1^+\Delta \wt P_i^+}e 
\underrightarrow{\wt P_i^+\Delta\wt Q_1^+}
 -\underrightarrow{ \iota \wt Q_1^+} + \tilde x_1^-+ \tilde x_2^-& \text{if} \quad \tilde e_i^-=u_i^-v_i^-,\\
   \iota\tga_{e_i}^-=\tilde e_i^++
%+\sum_{e\in \wt  Q_1^+\Delta \wt P_i^+}e 
\underrightarrow{\wt P_i^+\Delta \wt Q_1^+}
-\underrightarrow{ \iota\wt Q_1^+} + \tilde x_1^- & \text{if} \quad \tilde e_i^-=u_i^-v_i^+,
\end{cases}
\] 

We show the computation when $\tilde e_i^+=u_i^+v_i^+$ and $\tilde e_i^-=u_i^-v_i^-$, the other case is analogous. 
In this case, the cycle $\beta_{f_1}$ corresponds to 
\[
\beta_{f_1}=2\underrightarrow{\iota\wt Q_1^+}-2\underrightarrow{\wt Q_1^+}+\tilde x_1^+- \tilde x_1^-+ \tilde x_2^+- \tilde x_2^-.
\]
%The cycle $\beta_{f_1}$ intersects positively the edge $\tilde f_1^-$ by definition, therefore it will pass through the path $\wt Q_1^+$ with the same orientation as in $\tga_{e_i}^-$ and opposite orientation through $\iota\wt Q_1^+$.
Denoting by $\beta_{e_i}=\tga_{e_i}^--\iota\tga_{e_i}^-$, let us now compute  $\beta_{f_1}-\beta_{e_i}$.

\begin{equation}
    \begin{split}
        \beta_{f_1}-\beta_{e_i}&=2\underrightarrow{\iota\wt Q_1^+}-2\underrightarrow{\wt Q_1^+}+\tilde x_1^+- \tilde x_1^-+ \tilde x_2^+- \tilde x_2^-+\\
        &-(\tilde e_i^-+
\underrightarrow{\iota\wt P_i^+\Delta\iota \wt Q_1^+}
%\sum_{e\in \iota\wt Q_1^+\setminus \iota\wt P_i^+ }e+\sum_{e\in \iota\wt P_i^+\setminus \iota\wt Q_1^+ }e
 -\underrightarrow{\wt Q_1^+} + \tilde x_1^++ \tilde x_2^+)+\\
 &+\tilde e_i^++
 %\sum_{e\in \wt Q_1^+\setminus \wt P_i^+ }e+\sum_{e\in \wt P_i^+\setminus \wt Q_1^+ }e
%\sum_{e\in \wt Q^+\Delta \wt P_i^+ }e
\underrightarrow{\wt P_i^+\Delta \wt Q_1^+}
 -\underrightarrow{\iota\wt Q_1^+} + \tilde x_1^-+ \tilde x_2^-=\\
 & =-\underrightarrow{\wt Q_1^+}+\underrightarrow{\iota\wt  Q_1^+}-%\sum_{e\in \iota\wt Q_1^-\setminus \iota\wt P_i^+ }e-\sum_{e\in \iota\wt P_i^+\setminus \iota\wt Q_1^+ }e
%\sum_{e\in \iota\wt Q_1^+\Delta \iota\wt P_i^+ }e+
\underrightarrow{\iota\wt P_i^+\Delta\iota \wt Q_1^+}
+\underrightarrow{\wt P_i^+\Delta \wt Q_1^+}
%\sum_{e\in \wt Q_1^+\Delta \wt P_i^+ }e
-\tilde e_i^-+\tilde e_i^+%-\sum_{e\in\wt Q_1^+\setminus \wt P_i^+ }e+\sum_{e\in \wt P_i^+\setminus \wt Q_1^+ }e
=\\
&=\underrightarrow{\wt P_i^+} -\underrightarrow{\iota\wt P_i^+}-  \tilde e_i^-+\tilde e_i^+= \tga_{e_i}^+-\iota \tga_{e_i}^+ 
    \end{split}
\end{equation}

%\giusi{scrivi che l'intersezione è percorsa in direzioni diverse e quindi la differenza simmetrica la fai senza contare l'orientazione.}

\item 
The pair $(n_d,k)=(1,1)$:
since there is only one $\pi$-parallel loop that we call $x_1$, the cycles are of the form:

\[
\begin{cases}
    \tga_{e_i}^-=\tilde e_i^-+
%\sum_{e\in \iota\wt Q_1^+\Delta \iota\wt P_i^+ }e 
\underrightarrow{\iota\wt P_i^+\Delta\iota \wt Q_1^+}
 -\underrightarrow{\wt Q_1^+} + \tilde x_1^+& \text{if} \quad \tilde e_i^+=u_i^+v_i^+\\
   \tga_{e_i}^-=\tilde e_i^-+
%\sum_{e\in \iota \wt Q_1^+\Delta \iota\wt P_i^+ }e
\underrightarrow{\iota\wt P_i^+\Delta\iota \wt Q_1^+}
-\underrightarrow{\wt Q_1^+} & \text{if} \quad \tilde e_i^+=u_i^+v_i^-
\end{cases}
\] 
and
\[
\beta_{f_1}=2\underrightarrow{\iota\wt Q_1^+}-2\underrightarrow{\wt Q_1^+}+\tilde x_1^+- \tilde x_1^-.
\]
The computation is analogous.

\item The pair $(n_d,k)=(2,0)$: the cycle 
$
 \tga_{e_i}^-=\tilde e_i^-+
%\sum_{e\in \iota \wt Q_1^+\Delta \iota\wt P_i^+ }e
\underrightarrow{\iota\wt P_i^+\Delta\iota \wt Q_1^+}
-\underrightarrow{\wt Q_1^+}
$
and 
$
\beta_{f_1}=2\underrightarrow{\iota\wt Q_1^+}-2\underrightarrow{\wt Q_1^+},
$
and the result follows as in the previous case.

\end{itemize}

Therefore, we have found a $\ZZ$-linear combination of cycles $\beta_{f_1}-\beta_{e_i}$ in $\mathcal{B}_{\wt T}$ that gives rise to the cycle $\tilde\gamma_{e_i}^+-\iota \tilde\gamma_{e_i}^+$, when  $|E^{\wt T}_{\pi}|=1$.
\medskip
Suppose now $n_d-1+k\geq1$,
we will prove that the above formula generalises to: 
 \begin{equation}\label{eq:generalequationcombination}
      (\tga_{e_i}^+-\iota\tga_{e_i}^+)+(\tga_{e_i}^--\iota\tga_{e_i}^-)= \sum_{f_j\in E^{\wt T}_{\pi}}-\langle\tga_{e_i}^+,f_j^+\rangle\cdot \beta_{f_j},
 \end{equation}

that implies the linear combination in \cref{eq:generalequationombination}.
Suppose that the cycle $\ga_{e_i}$ intersects some $f_j\in E^{\wt T}_{\pi}$, for $2\leq j\leq n_d-1+k$. 
Then, the cycle $\tilde \ga_{e_i}^-=\wt T\cup \{\tilde e_i^-\}$ cannot pass through $\iota \wt P_i^+$, because $\iota \wt P_i^+$ contains the edges $\tilde f_j^-$ that are not in $\wt T$.
Thus, in $\pi^{-1}(T)^-$ it will pass through the edges in $\iota\wt P_i^{+}\Delta (\iota \wt Q_1^+\cup\dots\cup \iota\wt  Q_j^+)$.
Note that, when $m_d=0$, if the path $P_i$ in $T$ contains more than one subpath $Q_j$, say at least two $Q_{j_1},Q_{j_2}$, then there must exist another path $Q_{j_3}$ in $T$ connecting the endpoint of $Q_{j_1}$ to that of $Q_{j_2}$. Indeed, the endpoints of that path  belong again to $V(\Ga_{\dil})\cup V(P^{\wt T}_{\pi})$ and if there was no edge in $E^{\wt T}_{\pi}\cap Q_{j_3}$, then the preimage $\pi^{-1}(Q_{j_3})$ would necesarily contain a cycle since it cannot be of genus $0$ in case of  no dilated edges.
Therefore, the edges in $P_i$ but not in $Q_1\cup\dots \cup Q_j$ are  paths adjacent to the edge $e_i$ not in $T$. %\giusi{non so se è chiaro cosa intendo: voglio dire che $P_i$ può contenere lati non in $Q_j$ ma quei lati saranno tutti "esterni", cioè adiacenti al lato $e_i$  che chiude il ciclo. Non ci può essere un lato in mezzo a due $Q_i$ che non fa parte di un $Q_j$. } 

Moreover, in order to be a simple cycle, in $\pi^{-1}(T)^+$ it has to pass through the unique paths  $\wt Q_1^+\cup\dots\cup \wt Q_j^+$ with additional $\tilde x_i^+,\tilde x_j^+$ depending on whether the edges in $P_i\Delta (Q_i\cup\dots\cup Q_j)$ are attached to vertices
in $V(P^{\wt T}_{\pi})$.
 Again, we denote by $Q_j$  the unique path in $T$ containing $f_j\in E^{\wt T}_{\pi}$, as in \cref{orientedpath}.
The coefficients of the edges $\tilde f_j^{+}$ that appear in the cycle $\tga_{e_i}^-$ are prescribed by the cycle $\tga_{e_i}^+$:
\[
\langle\tga_{e_i}^-,\tilde f_j^{+}\rangle=\langle\tga_{e_i}^+,\tilde f_j^{+}\rangle=\pm 1,
\]
therefore,

\[
\langle(\tga_{e_i}^+-\iota\tga_{e_i}^+)+(\tga_{e_i}^--\iota\tga_{e_i}^-),\tilde f_j^{+}\rangle=\pm 2.
\]
%\[\langle(\tga_{e_i}^+-\iota\tga_{e_i}^+)+(\tga_{e_i}^--\iota\tga_{e_i}^-),\tilde f_i^-\rangle=\mp 2\]

Now, the cycle $\beta_{f_j}$ satisfies 
$\langle\beta_{f_j},\tilde f_j^{-} \rangle=2$ by definition, then 
the intersection with $\tilde f_j^+$ is
$\langle\beta_{f_j},\tilde f_j^{+} \rangle=-2$. It follows that in order to have that linear combination it is necessary to take into account also the parity given by the coefficient $\langle\tga_{e_i}^+,\tilde f_j^{+}\rangle$. Indeed,
\[
\langle\alpha\beta_{f_j},\tilde f_j^{+}\rangle = \langle(\tga_{e_i}^+-\iota\tga_{e_i}^+)+(\tga_{e_i}^--\iota\tga_{e_i}^-),\tilde f_j^{+}\rangle\quad \iff \alpha=-\langle\tga_{e_i}^+,\tilde f_j^{+}\rangle=\mp 1.
\]

Hence, any edge in $\wt Q_j^+$ and $\iota \wt Q_j^+$ that intersects the anti-symmetric cycle
has the exact multiplicity on both sides of \cref{eq:generalequationombination}.
Moreover, the preimage of each edge $e\in P_i$ that does not belong to any $\wt Q_j^+$, $\iota \wt Q_j^+$, will be contained in both $\tga_{e_i}^+$ and $\tga_{e_i}^-$ with opposite orientation and they sum up to zero. We have proved that any edge in the anti-symmetric cycle $(\tga_{e_i}^+-\iota\tga_{e_i}^+)+(\tga_{e_i}^--\iota\tga_{e_i}^-)$
appears in 
$ \sum_{j=1}^{n_d-1+k}-\langle\tga_{e_i}^+,\tilde f_j^{+}\rangle\cdot \beta_{f_j}$ with the same orientation and coefficient. For the other inclusion, there may be edges in the intersection of the paths $\wt Q^+_j$, $\iota\wt Q_j^+$ that  appear with a coefficient $2$ in each cycle $\beta_{f_j}$ that intersect. However,
they do not appear in the cycles $\tga_{e_i}^{\pm}$, because otherwise the cycle would pass through the same edge twice and the cycles are the smallest ones contained in $\wt T\cup \{\tilde e_i^{\pm}\}$. 
Thus, there may be more edges on the right hand side of \cref{eq:generalequationombination} and  must sum up to zero in order to have that equality.
Let us prove that, if $\langle \tga_{e_i}^+,\tilde f_1^+\rangle=\langle \tga_{e_i}^+,\tilde f_2^+\rangle$ they have opposite sign, otherwise they have same sign.

%\giusi{vorrei dire qualcosa del tipo: ho due cammini $Q_1$ e $Q_2$ con un lato in comune, quindi c'è una biforcazione. E cioè c'è un vertice in comune con due cammini  e se il nosto ciclo $\ga_{e_i} $ passa per entrambi i lati $f_1,f_2$ positivamente, quel vertice sarà pozzo per uno dei due lati, e.g. in $Q_1$ e sorgente per l'altro lato in $Q_2$, quindi il lato in comune sarà percorso in senso opposto dai cammini $Q_1$ e $Q_2$ di cui fa parte, perché i due cammini anche percorrono positivamente quei due lati $f_1,f_2$. Viceversa, se i due lati in questione, sono percorsi dal ciclo con segno opposto, allora quel vertice sarà pozzo o sorgente per entrambi i cammini e concludiamo.}

Suppose we have two paths $Q_i$, $Q_j$ that share a common edge $e$, in particular $e\not \in E^{\wt T}_{\pi}$ because there is exactly one edge in each path. Then at least one of the vertices of $e$ is contained in both paths.
If our cycle $\ga_{e_i}$
passes through $f_i,f_j$ with same positive orientation, then the common vertex
will be a sink for one of the paths (say $Q_i$) and a source for $Q_j$. Consequently the common edge is traversed in opposite orientation by the paths $Q_i, Q_j$, since both paths are also oriented positively along $f_i,f_j$ by definition.
Conversely, if the two edges 
 are traversed by the cycle with opposite signs, then the common vertex is either a sink or a source for both paths simultaneously, and the claim follows.

Let us show for completeness the computation when $k=0$ and there are only $n_d+1=j$ dilated vertices:

\begin{equation*}
    \begin{split}
        (\tga_{e_i}^+-\iota \tga_{e_i}^+)+(\tga_{e_i}^--\iota \tga_{e_i}^-)=& \tilde e_i^+-\tilde e_i^-+\tilde e_i^-+\tilde e_i^++\underrightarrow{\wt P_i^+}- \iota \underrightarrow{\wt P_i^+}+\underrightarrow{\iota \wt P_i^+\Delta( \iota \wt Q_1^+\cup\dots\cup \iota \wt Q_j^+)}+\\
&+\langle\tga_{e_i}^-,\tilde f_1^+\rangle\underrightarrow{\wt Q_1^+}+\dots+\langle\tga_{e_i}^-,\tilde f_j^+\rangle\wt Q_j^+
        - \underrightarrow{ \wt P_i^+\Delta( \wt Q_1^+\cup\dots\cup  \wt Q_j^+)}+\\&-\langle\tga_{e_i}^-,\tilde f_1^+\rangle\underrightarrow{\iota \wt Q_1^+}-\dots-\langle\tga_{e_i}^-,\tilde f_j^+\rangle\underrightarrow{\iota \wt Q_j^+}=\\=&-\langle\tga_{e_i}^-,\tilde f_1^+\rangle2\underrightarrow{\iota\wt Q_1^+}-\dots- \langle\tga_{e_i}^-,\tilde f_j^+\rangle2\underrightarrow{\iota\wt Q_j^+}+\\
        &+
        \langle\tga_{e_i}^-,\tilde f_1^+\rangle 2\underrightarrow{\wt Q_1^+}+\dots+\langle\tga_{e_i}^-,\tilde f_j^+\rangle 2\wt Q_j^+=\\
        =&-\langle\tga_{e_i}^+,\tilde f_1^+\rangle (2\iota\underrightarrow{\wt Q_1^+}-2\underrightarrow{\wt Q_1^+})- \dots -\langle\tga_{e_i}^+,\tilde f_j^+\rangle (2\iota\underrightarrow{\wt Q_j^+}-2\underrightarrow{\wt Q_j^+})=\\
=&\sum_{j=1}^{n_d-1+k}-\langle\tga_{e_i}^+,\tilde f_j^{+}\rangle\cdot \beta_{f_j}.
    \end{split}
\end{equation*}

\end{proof}

    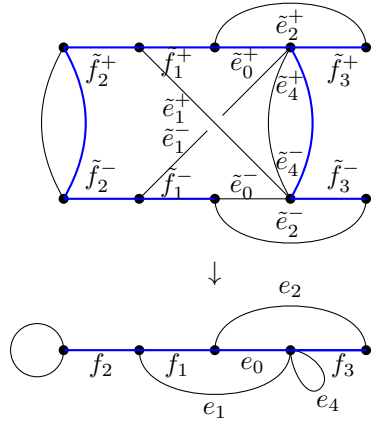
\begin{figure}[H]
    \centering
\begin{tikzpicture}
%primo vertice a sx
%\coordinate (-1) at (-2,0);
%\coordinate (-1a) at (-2,2);
%\coordinate (-1b) at (-2,3);
\coordinate (0) at (-1,0);
%\coordinate (0a) at (-1,2.5+0.5);

\coordinate (1) at (0,0); %leftvertex
 \coordinate (a1) at (0,2); %preimage - of the vertex (1)
\coordinate (b1) at (0,4); %preimage + of the vertex (1)
\coordinate (2) at (1+1,0); %rightvertex
 \coordinate (a2) at (1+1,2); %preimage - of the vertex (2)
 \coordinate (b2) at (1+1,4); %preimage + of the vertex (2)
 \coordinate (a0) at (-1,2);
\coordinate (b0) at (-1,4);

 \coordinate (12) at (1,0);
 \coordinate (12a) at (1,2);
 \coordinate (12b) at (1,4);

 \coordinate (3) at (2+1,0);
 \coordinate (a3) at (2+1,2);
\coordinate (b3) at (2+1,4);

 \coordinate (4) at (2.2+1,0);
  \coordinate (a4) at (2.2+1,2);
  \coordinate (b4) at (2.2+1,5);
 \coordinate (c1) at (1.4-0.5,2.4+0.5); 
 \coordinate (c2) at (1.6-0.5,2.6+0.5);
 %contruction of the edge in the middle: c1 and c2 are the middle points so that the segment is cutted in half

 \foreach \i in {0,a0,b0,3,a3,12,12a,12b,b3,2,1,a2,a1,b2,b1}
 \draw[fill=black](\i) circle (0.15em);

%\draw[line width=1.5pt,black](0) circle (0.50em);
%\draw[line width=1.5pt,black](0a) circle (0.50em);
%\draw[line width=1.5pt,black](4) circle (0.50em);
%\draw[line width=1.5pt,black](a4) circle (0.50em);
%\draw[thin,black](-1) circle (0.50em);
%\draw[thin,black](-1a) circle (0.50em);
\draw[thin,black](-1.35,0) circle (1em);

\draw[thin, black] (0)--(1);
\draw[thin, black] (2)--(1);
\draw[thin, black] (a2)--(a1);
\draw[thick, blue] (b2)--(b1);

% \draw[thin, black] (1) to [out=90, in=90] (12);  
\draw[thin, black] (12) to [out=90, in=90] (3);
\draw[thin, black] (12b) to [out=90, in=90] (b3);
\draw[thin, black] (1) to [out=90+180, in=-90] (2);

%loop

 \draw[thin, black] (2) to [out=-90, in=180+45] (2.4,-0.5);
  \draw[thin, black] (2) to [out=0, in=45] (2.4,-0.5);

  %preimm loops
  \draw[thick, blue] (b2) to [out=-60, in=60] (a2);
  \draw[thin, black] (b2) to [out=180+60, in=90+30] (a2);

   \draw[thick, blue] (b0) to [out=-60, in=60] (a0);
  \draw[thin, black] (b0) to [out=180+60, in=90+30] (a0);

%\draw[thin, black] (b1) to [out=90, in=90] (12b);\draw[thin, black] (12b) to [out=90, in=90] (b2);
\draw[thin, black] (12a) to [out=90+180, in=-90] (a3);
%\draw[thin, black] (a1) to [out=90+180, in=-90] (12a);

\draw[thick, blue] (0)--(3);
\draw[thick, blue] (a1)--(a0);
\draw[thick, blue] (a1)--(12a);
\draw[thick, blue] (b1)--(b0);
  \draw[thin, black] (b1)--(a2);
\draw[thin, black] (a1)--(c1);
    \draw[thick, blue] (a3)--(a2);
    \draw[thick, blue] (b3)--(b2);
    \draw[thick, blue] (2)--(3);
    \draw[thin, black] (b2)--(c2);

  \draw (1,1.3) node [below] {$\downarrow$};
    %\draw (0) node  {$1$};
     %\draw (0a) node  {$1$};
  %\draw (4) node  {$1$};
   %\draw (a4) node  {$1$};
    \draw (1,-0.8) node  {$e_1$};
     \draw (2,0.8) node  {$e_2$};
     \draw (2,4.3) node  {$\tilde e_2^+$};
     \draw (2,1.7) node  {$\tilde e_2^-$};
     \draw (-0.5,-0.2) node  {$f_2$};
     \draw (2.7,-0.2) node  {$f_3$};
      \draw (2.5,-0.7) node  {$e_4$};
    \draw (1.5,-0.2) node  {$e_0$};
    %\draw (1.5,0.5) node  {$e_3$};
    \draw (0.5,-0.2) node  {$f_1$};
    %\draw (0.5,0.5) node  {$e_2$};
      \draw (0.5,3.2) node  {$\tilde e_1^+$};
        \draw (0.5,2.8) node  {$\tilde e_1^-$};

         \draw (-0.5,3.7) node  {$\tilde f_2^+$};
        \draw (-0.5,2.3) node  {$\tilde f_2^-$};
        \draw (2.7,3.7) node  {$\tilde f_3^+$};
        \draw (2.7,2.3) node  {$\tilde f_3^-$};
 \draw (2,3.5) node  {$\tilde e_4^+$};
        \draw (2,2.5) node  {$\tilde e_4^-$};
        
    \draw (1.4,3.8) node  {$\tilde e_0^+$};
    %\draw (1.4,3.8+0.8) node  {$\tilde e_3^+$};
     \draw (1.4,2.2) node  {$\tilde e_0^-$};
    % \draw (1.4,2.2-0.75) node  {$\tilde e_3^-$};
    \draw (0.5,3.8) node  {$\tilde f_1^+$};
    % \draw (0.5,3.8+0.8) node  {$\tilde e_2^+$};
     \draw (0.5,2.2) node  {$\tilde f_1^-$};
 %\draw (0.5,2.2-0.75) node  {$\tilde e_2^-$};
  \end{tikzpicture}
 \caption{A free double cover with two $\pi$-parallel loops in $P^{\wt T}_{\pi}.$ The admissible spanning tree $\wt T$ is highlighted in blue.}
    \label{ex:condition+-}
\end{figure}

In what follows we will show the relation with the basis already constructed in \cite{RZ_ngonal}. We have seen that the cycles $\beta_e$ are obtained by choosing a spanning tree and then making a choice for the admissible spanning tree of the cover graph. We will prove that one of these choices gives rise to the anti-symmetric cycles described in \cref{eq:Prymcbasis2} by Röhrle and Zakharov. Therefore, \cref{construction} gives an intrinsic way of defining the basis of anti-symmetric cycles.

\begin{prop}\label{prop:relationtwobases}
Let $\pi\colon \tGa\to\Ga$ be a harmonic double cover of metric graphs and let $T$ be a spanning tree of $\Ga$.
There exists an admissible spanning tree $\wt T$ with respect to $T$ such that the basis
$\mathcal B=\{\tilde\alpha_1^{+}- \tilde\alpha_1^-,\dots,\tilde\alpha_A^+-\tilde\alpha_A^-,2\tilde\beta_1,\dots,2\tilde\beta_B\}$ described in (\ref{eq:Prymcbasis2})
    coincides with $\mathcal{B}_{\wt T}$, up to a sign.
	\begin{proof}
		Suppose $\pi$ is free, then the basis $\mathcal B$ can be obtained by choosing $\wt T=\pi^{-1}(T)^+\cup \pi^{-1}(T)^-\cup \tilde e_g^+$ where $\tilde e_g^+$ traverses the two trees, up to a sign given by the orientation used in Construction $1$ which is different from the one used in \cite{RZ_ngonal}. 
        If $\pi$ has dilated,
		recall that in \cite{RZ_ngonal}, the autors first contract the dilation subgraph $\Ga_{\dil}$ into isolated dilated vertices $v_0',\dots,v'_{d-1}$ resulting a double cover $\tGa'\to \Ga'$, where $d$ is the number of connected components of $\Ga_{\dil}$, and then consider a second double cover $\wt \Ga''\to\Ga''$ obtained by replacing the dilated vertices with $\pi$-parallel loops. In this way, given a spanning tree $T''$ of $\Ga''$, the preimage $\pi^{-1}(T'')$ is always disconnected. %Contrary to \cref{construction}, where the preimage of any spanning tree is connected and can contain dilated edges.
        Running \cref{construction} on $\tGa'\to\Ga'$, using the same spanning tree $T'$, any admissible tree $\wt T'$ of $\tGa'$ is made by a choice of an edge  $e'\in T'$ such that $\tilde e'^-\not \in \wt T$. Between the two dilated vertices $v'_i$ and $v'_{i+1}$, choose such an edge $e_i'$ to be  adjacent to $v'_{i}$ and not a loop, and define
		\[
		\wt T'=\pi^{-1}(T)^+\cup\pi^{-1}(T)^-\setminus\cup_i\{\tilde e_i'^-\},
		\]

		in this way the cycle $\beta_{e'}$  corresponds to the cycle  $2\beta_i$ in \cite[Proposition 4.20]{RZ_ngonal}. Since there are no dilated edges in this double cover and $P^{\wt T'}_{\pi}=\emptyset$, the number of such cycles is $-m_d^T+n_d-1=n_d-1=B$.
		The rest of the cycles are done  in the same way by the two constructions up to a choice of orientation.
		And since \cref{construction} doesn't take into account the dilated edges, the basis for anti-symmetric cycles for $\wt\Ga'$ is the same as the basis for anti-symmetric cycles for $\tGa$.
	\end{proof}
\end{prop}

	\section{Continuity of the Prym construction}\label{sec:prymcontinuous}
Finally, we are ready to prove that it is possible to choose the basis of anti-symmetric cycles in such a way  that the Prym construction is continuous under edge contractions. 

Suppose that $\pi:\wt \Gamma\to \Gamma$ is obtained from $\pi':\wt \Gamma'\to \Gamma'$ by contracting a subgraph $S\subset \Gamma'$.
	Let $\wt S:=\pi'^{-1}(S)\subset \wt \Gamma$ be the preimage of $S$ under $\pi'$.
	In this case we have the following commutative diagram of metric graphs 
	
\begin{center}
\begin{tikzcd}
	\widetilde{\Gamma}' \arrow[rr, "h_{\wt S}"] \arrow[dd, "\pi'"] &  & \widetilde{\Gamma}'/\wt S =\tGa \arrow[dd, "\pi"] \\
	&  &                   \\
	\Gamma' \arrow[rr, "h_S"]                 &  &   \Gamma'/S =\Ga            
\end{tikzcd}
\end{center}

	where with $h_S$ (and $h_{\wt S}$) we denote the contraction map of $S$ (and $\wt S$).
    We will also call $h_S$ (and $h_{\wt S}$)
    the contraction morphisms on the underlying graphs by abuse of notation and denote with $h_S^E$  (resp. $h_{\wt S}^E$) the inclusion map on the edges $h_S^E: E(\Gamma'/S)\hookrightarrow E(\Gamma')$  induced by  $h_S$ (resp. $h_{\wt S}$). We have the following.

\begin{thm}\label{continuityofbasis}
	
	 Let $\pi:\wt \Gamma\to \Gamma$ be obtained from $\pi':\wt \Gamma'\to \Gamma'$ by contracting a subgraph $S\subset \Gamma'$ and let $\wt T$ be an admissible spanning tree of $\tGa$.
	There exists an admissible spanning tree $\wt T'$ for $\wt \Ga'$ such that $h_{\wt S}(\wt T')=\wt T$, where $h_{\wt S}$ is the contraction morphism.
    In particular, the basis of anti-symmetric cycles $\mathcal{B}_{\wt T'} $
    associated to the double cover $\pi'$ gets contracted to $\mathcal B_{\wt T}$. %In other words, the construction of the basis is continuous.
	
	\begin{proof}
		Let us consider first the case $S=\{e\}$, where $e$ is not a loop.

	%Denote with. Moreover, denote by $h_e$ the induced map on the graphs $h_e\colon E(\Ga)\cup V(\Ga)\to E(\Ga/e)\cup V(\Ga/e)$such that if $e=xy$ with $x,y\in V(\Ga)$, then $h_e^{-1}(x=y)=e$.

	Consider an orientation on $\Ga$ and a spanning tree $T$ of $\Gamma$. 
	Observe that $h^E_S(T)$ is disconnected and acyclic   and let  $T':=h^E_S(T)\cup e$. It is a spanning tree of $\Gamma'$ because $e$ is not a loop. 
	%Moreover, we can assume that $T'$ is obtained from $(h_e^E)^{-1}(T)$ by adding the edge $e$. 
	Assume that the orientation on $\Gamma$ is induced by the orientation on $\Gamma'$.	
	
	Consider now an admissible spanning tree $\wt T$ of $\wt\Gamma$ satisfying the properties of Construction \ref{construction} and construct the basis for anti-symmetric cycles $\mathcal B_{\wt T}$ as defined in \cref{def:basis}.

	%Consider a spanning tree $T_S$ of $S$ and a spanning tree $T$ of $\Gamma$ such that $T':=T_S\cup T$ is a spanning tree of $\Gamma'$. Let also $\wt T$ be a spanning tree of $\tGa$ as in \cref{construction}.
	
	We want to show that we can choose $\widetilde T'$ an admissible spanning tree of $\tGa'$ whose contraction gives $\widetilde T$. This will imply that the basis for anti-symmetric cycles $\mathcal B_{\wt T'}$ of $\widetilde \Gamma'$ will be contracted to the basis for anti-symmetric cycles $\mathcal B_{\wt T}$ of $\tGa$.
	
%	Suppose that $\tilde S$ is not connected. Then $\wt T_S:=\pi^{-1}(T_S)$ is a spanning forest of $\tilde S$ and  $\wt T':=\wt T\cup \wt T_S$ is a spanning tree of $\wt \Gamma'$. Moreover $\pi_{\tilde S}$ sends $\wt T'$ to $\wt T$.	One can check (essentially by construction) that if we repeat \cref{construction} for $\pi'$ and with respect to this tree, that the basis for anti-symmetric cycles that we construct will degenerate to the one associated to $\pi$ and $T$. 

	Choose \[\wt T'=h_{\wt S}^{E}(\wt T')\cup \wt S.\]
    It is always a spanning tree of $\tGa'$ since  $\wt S$, which corresponds to the set $\pi^{-1}(e)$, does not contain loops.
    Construct $\mathcal B_{\wt T'}$ the basis of anti-symmetric cycles on $\tGa'$. %We will then prove  that $\mathcal B_{\wt T'}$ contracts to $\mathcal B_{\wt T}$. 

Let us prove that the basis for anti-symmetric cycles $\mathcal B_{\wt T'}$ contracts to $\mathcal B_{\wt T}$.
  Since  $\wt S$ belongs to the spanning tree $\wt T$, then $S\not \in E_{nz}$ 
  and therefore
  the anti-symmetric cycles $\beta'_e\in \mathcal B_{\wt T'}$ satisfy $e\not \in S$. Moreover, any cycle that contains the edges in $\wt S$ will be sent under the contraction morphism $h_{\wt S}$ to the cycle in which their respective lengths become zero. The cycle $h_{\wt S}(\beta'_e)$ corresponds to  $\beta_e\in\mathcal B_{\wt T}$ by construction.
	  It follows that all the anti-symmetric cycles constructed for $\pi:\wt \Gamma\to \Gamma$ are limits of anti-symmetric cycles for $\wt \Gamma'\to\Gamma'$.

Suppose now that $S=\{e\}$ and $e$ is a loop. Then, $h_S^E(T)$ is already a spanning tree of $\Ga'$ because the contraction of the loop does not decrease the number of vertices.

The preimage $\pi^{-1}(e)$ can be one of the following types: if $e$ is not dilated, 
it consists of two parallel edges or  two loops,  otherwise $\pi^{-1}(e)$ consists of one dilated loop. 
  We are left with choosing the admissible spanning tree $\wt T'$ that contracts to $\wt T$. 
If the preimage consists of one or two loops, then 
choose $\wt T'=h_{\wt S}^E(\wt T)$ as spanning tree otherwise, since the contraction of two parallel edges $\tilde e^{\pm}$ decreases by one the number of vertices, take as spanning tree 
$h_{\wt S}^E(\wt T)$ together with $\tilde e^+$:
\[
\wt T'= h^E_{\wt S}(\wt T)\cup \tilde e^+,
\]
it is an admissible spanning tree by condition $(b)$ in \cref{construction}.

	Suppose now that $|S|>1$. We can repeat the discussion above by contracting each edge in $S$ at a time. %We have then two different situations: either $\wt T_S$ is a spanning subgraph of $\wt S$, or it is not.In the first case, the dilation cycle must intersect $S$. In the second case, there are no dilated vertices or edges in $S$ (and $S$ will be contracted in a dilated vertex of $\Gamma$.	Assume that we are in this second situation and observe that we must further consider two different situations: either $\wt T_S\cup \wt T$ is a spanning tree of $\wt \Gamma'$, or it is not.

Using the same argument as before it  is possible to show that $\mathcal B_{\wt T'}$
  contracts to $\mathcal B_{\wt T}$.
        Moreover, the intersection pairing between two anti-symmetric cycles is preserved under contraction. Indeed, the anti-symmetric cycles that contain the edges in $ \wt S$ will have now intersection $0$ on those edges, since the length $\ell(e)\to0$, while the intersection
        on the edges that are not affected by the contraction will be the same.
        
        %\[
       % \langle h_{\wt S}(\beta'_e),  ?\rangle 
       % \]

\end{proof}
\end{thm}

\begin{rmk}
	Notice that we have proved that this particular basis beahaves well under edge contractions, but the construction of the Prym variety is independent of the choice of the basis. Therefore, we will prove that actually the extended $\overline{\mathrm{Prym}_c(\pi)}$ is continuous under the Prym map from the moduli space of harmonic double covers to the moduli space of pptavs.
\end{rmk}

\begin{ex} It is important to allow preimages of $\pi$-parallel loops in the admissible tree in order to have continuity when contracting those edges as in \cref{fig:contraction}. The choice of another spanning tree of $\tGa''$ would have not been contracted to the unique admissible spanning tree $\wt T$. 

In this picture, we depict two contraction morphisms of a harmonic double cover of a graph of genus $3$ and the admissible spanning trees at each step. In order to have the same basis of cycles in the limit, we have to choose the spanning tree $\wt T''=\{\tilde e_1^+\}\cup \{\tilde e_3^+\}\cup \{\tilde e_0^+\}\cup \{\tilde e_2^+\}$ that contracts to $\wt T'=\{\tilde e_1^+\}\cup \{\tilde e_3^+\}\cup \{\tilde e_2^+\}$ and finally contracts to the unique possible spanning tree $\wt T'=\{\tilde e_1^+\}\cup \{\tilde e_3^+\}$ in $\tGa$.
\begin{center}
    \begin{figure}[H]
    \begin{tikzpicture}
				
	\coordinate (1x) at (0,0);
	\coordinate (3x) at (-4,0);
	\coordinate (a1x) at (0,2);
    \coordinate (a1xx) at (0,5);
	\coordinate (a3x) at (-4,2);
	\coordinate (b1x) at (0,3.5);
	\coordinate (b3x) at (-4,5);
	\coordinate (2x) at (-2,0);
	\coordinate (4x) at (-2,0);
	\coordinate (a2x) at (-2,2);
	\coordinate (a4x) at (-2,2);
	\coordinate (b2x) at (-2,5);
	\coordinate (b4x) at (-2,5);
				
				\foreach \i in {1x,2x,4x,3x,a1x, a1xx,a2x,a4x,b2x,b4x,a3x,b3x}
				\draw[fill=black](\i) circle (0.15em);
				\draw[thick, blue] (3x)--(2x);
					\draw[thick, blue] (4x)--(1x);
				%\draw[thin, black] (a4x)--(-1.95,3.3);
				%\draw[thin, black] (-2.09,3.65)--(b2x);
				%\draw[thin, black] (b4x)--(a2x);

				\node[circle, draw, minimum size=1cm] (c) at (-0.5-4,0) {};
				%\node[circle, draw, minimum size=0.4cm] (c1) at (2-4+2.2,0) {};
               % \node[circle, draw, minimum size=0.4cm] (c1) at (2-4+2.5-0.5+0.2,3.5) {};
					\node[circle, draw, minimum size=1cm] (c3) at (-2,-0.5) {};
				\node[circle, draw, minimum size=1cm] (c3) at (0.5,0) {};
				\draw[thin, black] (a3x)--(a4x);
				\draw[thick, blue] (b3x)--(b4x);
			\draw[thin, black] (a1x)--(a4x);
            \draw[thin, blue] (a1xx)--(b4x);
				%\draw[thick, blue] (b1x)--(b4x);

				\draw (-2,1.4) node [below] {$\downarrow$};
                 \draw (1.8,1.5) node [below] {$\longleadsto$};
                  \draw (1.5,1.2) node [below] {$h_{e_4}$};
				%\draw (0.2,3.7) node [below] {$1$};
\draw (2-4+2.9,-0.2) node [right] {$e_4$};
\draw (-5.2,0) node [below] {$e_0$};
\draw (-4.6+0.6,3.5) node [right] {$\tilde e_0^+$};
\draw (-5+0.5,3.5) node [left] {$\tilde e_0^-$};
\draw (-5-1,5) node [left] {$\wt \Ga''$};
\draw (-5-1,0) node [left] {$\Ga''$};
\draw (-3.2,0) node [below] {$e_1$};
\draw (-3.2,5) node [above] {$\tilde e_1^+$};
\draw (-3.2,2) node [below] {$\tilde e_1^-$};
  \draw (-1,0) node [below] {$e_3$};
  \draw (-1,2) node [below] {$\tilde e_3^-$};
  \draw (-1,5) node [above] {$\tilde e_3^+$};
  \draw (-2,-1) node [below] {$e_2$};
   \draw (-1.9,3.5) node [right] {$\tilde e_2^+$};      
   \draw (-2,3.5) node [left] {$\tilde e_2^-$};
    \draw (-0.65,3.5) node [right] {$\tilde e_4^-$}; 
    \draw (0.6,3.5) node [right] {$\tilde e_4^+$}; 
    
	\draw[thick, blue] (b3x) to [out=-45, in=45] (a3x);
				\draw[thin, black] (b3x) to [out=180+45, in=90+45] (a3x);
				\draw[thick, blue] (b4x) to [out=-45, in=45] (a4x);
				\draw[thin, black] (b4x) to [out=180+45, in=90+45] (a4x);
                \draw[thick, blue] (a1xx) to [out=-45, in=45] (a1x);
				\draw[thin, black] (a1xx) to [out=180+45, in=90+45] (a1x);
			
					\end{tikzpicture}
        \begin{tikzpicture}
				
	\coordinate (1x) at (0,0);
	\coordinate (3x) at (-4,0);
	\coordinate (a1x) at (0,3.5);
	\coordinate (a3x) at (-4,2);
	\coordinate (b1x) at (0,3.5);
	\coordinate (b3x) at (-4,5);
	\coordinate (2x) at (-2,0);
	\coordinate (4x) at (-2,0);
	\coordinate (a2x) at (-2,2);
	\coordinate (a4x) at (-2,2);
	\coordinate (b2x) at (-2,5);
	\coordinate (b4x) at (-2,5);
				
				\foreach \i in {1x,2x,4x,3x,a1x,a2x,a4x,b2x,b4x,a3x,b1x,b3x}
				\draw[fill=black](\i) circle (0.15em);
				\draw[thick, blue] (3x)--(2x);
					\draw[thick, blue] (4x)--(1x);
				%\draw[thin, black] (a4x)--(-1.95,3.3);
				%\draw[thin, black] (-2.09,3.65)--(b2x);
				%\draw[thin, black] (b4x)--(a2x);

				\node[circle, draw, minimum size=1cm] (c) at (-0.5-4,0) {};
				\node[circle, draw, minimum size=0.4cm] (c1) at (2-4+2.2,0) {};
                \node[circle, draw, minimum size=0.4cm] (c1) at (2-4+2.5-0.5+0.2,3.5) {};
					\node[circle, draw, minimum size=1cm] (c3) at (-2,-0.5) {};
				
				\draw[thin, black] (a3x)--(a4x);
				\draw[thick, blue] (b3x)--(b4x);
			\draw[thin, black] (a1x)--(a4x);
				\draw[thick, blue] (b1x)--(b4x);

				\draw (-2,1.4) node [below] {$\downarrow$};
                 \draw (1.6,1.5) node [below] {$\longleadsto$ };
                  \draw (1.5,1.2) node [below] {$h_{e_0}$};
				\draw (0.2,3.7) node [below] {$1$};
\draw (2-4+2.2,0.2) node [below] {$1$};
\draw (-5.2,0) node [below] {$e_0$};
\draw (-4.6+0.6,3.5) node [right] {$\tilde e_0^+$};
\draw (-5+0.5,3.5) node [left] {$\tilde e_0^-$};
\draw (-5-1,5) node [left] {$\wt \Ga''$};
\draw (-5-1,0) node [left] {$\Ga''$};
\draw (-3.2,0) node [below] {$e_1$};
\draw (-3.2,5) node [above] {$\tilde e_1^+$};
\draw (-3.2,2) node [below] {$\tilde e_1^-$};
  \draw (-1,0) node [below] {$e_3$};
  \draw (-1,2.8) node [below] {$\tilde e_3^-$};
  \draw (-1,4.9) node [below] {$\tilde e_3^+$};
  \draw (-2,-1) node [below] {$e_2$};
   \draw (-1.9,3.5) node [right] {$\tilde e_2^+$};      
   \draw (-2,3.5) node [left] {$\tilde e_2^-$};
	\draw[thick, blue] (b3x) to [out=-45, in=45] (a3x);
				\draw[thin, black] (b3x) to [out=180+45, in=90+45] (a3x);
				\draw[thick, blue] (b4x) to [out=-45, in=45] (a4x);
				\draw[thin, black] (b4x) to [out=180+45, in=90+45] (a4x);
			
					\end{tikzpicture}
          \begin{tikzpicture}
				
	\coordinate (1x) at (0,0);
	\coordinate (3x) at (-4,0);
	\coordinate (a1x) at (0,3.5);
	\coordinate (a3x) at (-4,3.5);
	\coordinate (b1x) at (0,3.5);
	\coordinate (b3x) at (-4,3.5);
	\coordinate (2x) at (-2,0);
	\coordinate (4x) at (-2,0);
	\coordinate (a2x) at (-2,2);
	\coordinate (a4x) at (-2,2);
	\coordinate (b2x) at (-2,5);
	\coordinate (b4x) at (-2,5);
				
				\foreach \i in {1x,2x,4x,3x,a1x,a2x,a4x,b2x,b4x,a3x,b1x,b3x}
				\draw[fill=black](\i) circle (0.15em);
				\draw[thick, blue] (3x)--(2x);
					\draw[thick, blue] (4x)--(1x);
				%\draw[thin, black] (a4x)--(-1.95,3.3);
				%\draw[thin, black] (-2.09,3.65)--(b2x);
				%\draw[thin, black] (b4x)--(a2x);

				\node[circle, draw, minimum size=0.4cm] (c) at (-4-0.2,0) {};
                \node[circle, draw, minimum size=0.4cm] (c) at (-4-0.2,3.5) {};
				\node[circle, draw, minimum size=0.4cm] (c1) at (2-4+2.2,0) {};
                \node[circle, draw, minimum size=0.4cm] (c1) at (2-4+2.5-0.5+0.2,3.5) {};
					\node[circle, draw, minimum size=1cm] (c3) at (-2,-0.5) {};
				
				\draw[thin, black] (a3x)--(a4x);
				\draw[thick, blue] (b3x)--(b4x);
			\draw[thin, black] (a1x)--(a4x);
				\draw[thick, blue] (b1x)--(b4x);

				\draw (-2,1.4) node [below] {$\downarrow$};
				\draw (0.2,3.7) node [below] {$1$};
                \draw (-4.2,3.7) node [below] {$1$};
                \draw (-4.2,0.2) node [below] {$1$};
\draw (2-4+2.2,0.2) node [below] {$1$};
%\draw (-5.2,0) node [below] {$e_0$};
%\draw (-4.6+0.6,3.5) node [right] {$\tilde e_0^+$};
%\draw (-5+0.5,3.5) node [left] {$\tilde e_0^-$};
\draw (-3.2,0) node [below] {$e_1$};
\draw (-3.2,4.9) node [below] {$\tilde e_1^+$};
\draw (-3.2,2.8) node [below] {$\tilde e_1^-$};
\draw (1,1.5) node [below] {$\longleadsto$};
\draw (1.2,1.2) node [below] {$h_{e_2}$};
  \draw (-1,0) node [below] {$e_3$};
  \draw (-1,2.8) node [below] {$\tilde e_3^-$};
  \draw (-1,4.9) node [below] {$\tilde e_3^+$};
  \draw (-2,-1) node [below] {$e_2$};
   \draw (-1.9,3.5) node [right] {$\tilde e_2^+$};      
   \draw (-2,3.5) node [left] {$\tilde e_2^-$};
   \draw (-4.5,4.9) node [left] {$\wt \Ga'$};
\draw (-4.5,0) node [left] {$\Ga'$};
				
				\draw[thick, blue] (b4x) to [out=-45, in=45] (a4x);
				\draw[thin, black] (b4x) to [out=180+45, in=90+45] (a4x);
			
					\end{tikzpicture}
          \centering
      \begin{tikzpicture}[scale=0.95]
				
	\coordinate (1x) at (0,0);
	\coordinate (3x) at (-4,0);
	\coordinate (a1x) at (0,3.5);
	\coordinate (a3x) at (-4,3.5);
	\coordinate (b1x) at (0,3.5);
	\coordinate (b3x) at (-4,3.5);
	\coordinate (2x) at (-2,0);
	\coordinate (4x) at (-2,0);
	\coordinate (a2x) at (-2,3.5);
	\coordinate (a4x) at (-2,3.5);
	\coordinate (b2x) at (-2,3.5);
	\coordinate (b4x) at (-2,3.5);
				
				\foreach \i in {1x,2x,4x,3x,a1x,a2x,a4x,b2x,b4x,a3x,b1x,b3x}
				\draw[fill=black](\i) circle (0.15em);
				\draw[thick, blue] (3x)--(2x);
					\draw[thick, blue] (4x)--(1x);

				\node[circle, draw, minimum size=0.4cm] (c) at (-4-0.2,0) {};
                \node[circle, draw, minimum size=0.4cm] (c) at (-4-0.2,3.5) {};
				\node[circle, draw, minimum size=0.4cm] (c1) at (2-4+2.2,0) {};
                \node[circle, draw, minimum size=0.4cm] (c1) at (2-4+2.5-0.5+0.2,3.5) {};
					\node[circle, draw, minimum size=0.4cm] (c3) at (-2,-0.2) {};
				\node[circle, draw, minimum size=0.4cm] (c3) at (-2,3.3) {};

				\draw (-2,1.9) node [below] {$\downarrow$};
				\draw (0.2,3.7) node [below] {$1$};
                \draw (-2,3.5) node [below] {$1$};
                \draw (-2,0) node [below] {$1$};
                   \draw (-4.5,4) node [left] {$\wt \Ga$};
\draw (-4.5,0) node [left] {$\Ga$};
                \draw (-4.2,3.7) node [below] {$1$};
                \draw (-4.2,0.2) node [below] {$1$};
\draw (2-4+2.2,0.2) node [below] {$1$};

\draw (-3.2,0) node [below] {$e_1$};
\draw (0.2,-1.2) node [below] {$.$};
\draw (-3.2,4.7) node [below] {$\tilde e_1^+$};
\draw (-3.2,3) node [below] {$\tilde e_1^-$};

  \draw (-1,0) node [below] {$e_3$};
  \draw (-1,3) node [below] {$\tilde e_3^-$};
  \draw (-1,4.7) node [below] {$\tilde e_3^+$};
  
				\draw[thick, blue] (b1x) to [out=90+45, in=45] (a4x);
				\draw[thin, black] (b1x) to [out=180+45, in=-45] (a4x);
				\draw[thick, blue] (b3x) to [out=45, in=90+45] (a4x);
				\draw[thin, black] (b3x) to [out=-45, in=180+45] (a4x);
			
					\end{tikzpicture}
                \caption{Contraction of the edges $e_4$, $e_0$ and $e_2$ and the relative admissible spanning trees highlighted in blue.}
    \label{fig:contraction}
            \end{figure}
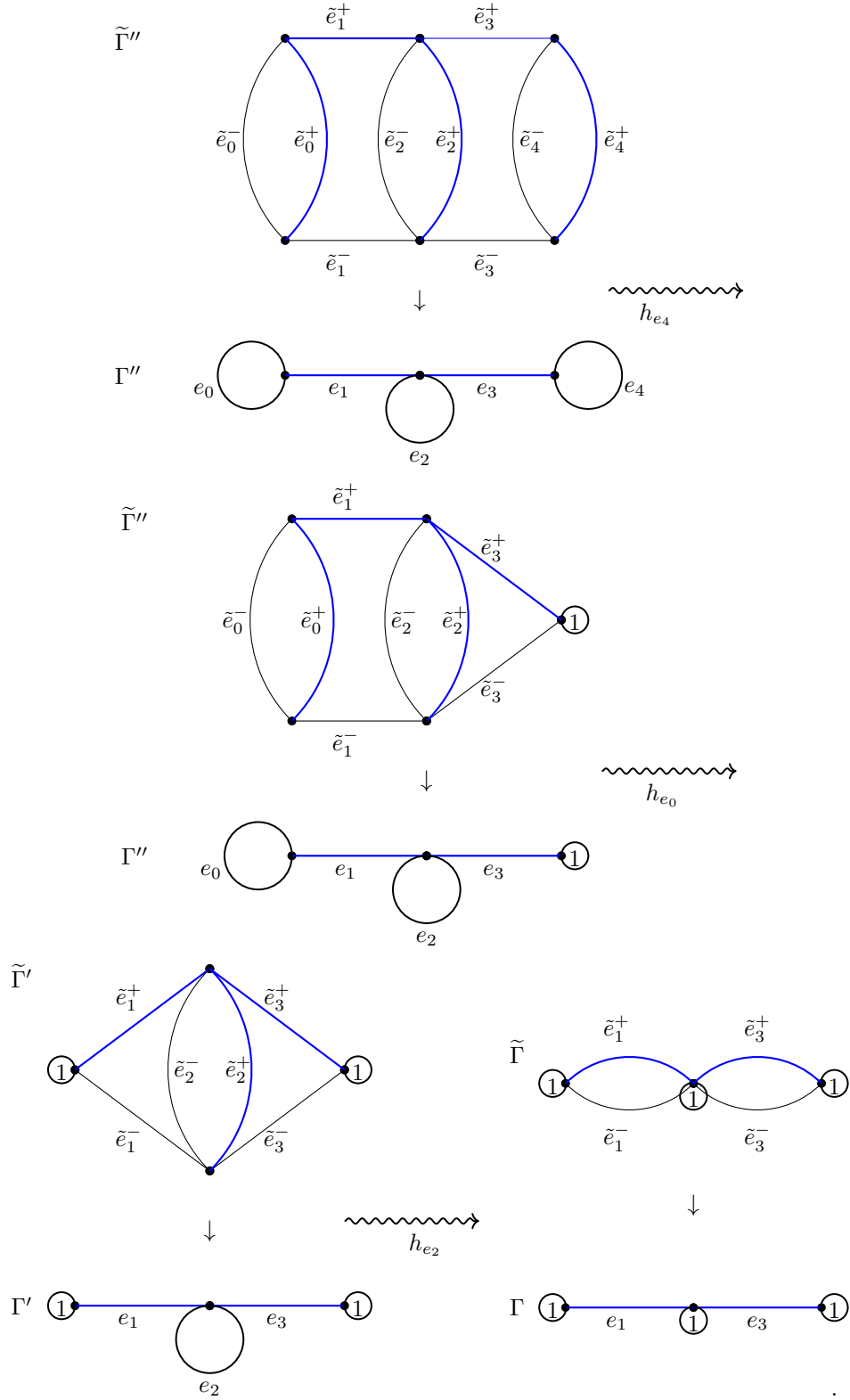
\end{center}
\end{ex}

Recall that the tropical Prym--Torelli map $P^{trop}_g\colon \mathcal R^{trop}_ {g} \to \mathcal A^{trop,V}_{g-1} $ assigns to a harmonic double cover $\pi\colon \tGa_{\tilde w}\to\Ga_w$ its   extended Prym variety $\overline{\Prym_c(\pi)}$ as in Definition \ref{def:prymtorelli}.  The extended Prym is a torus of dimension $g-1$ defined by a positive semi-definite quadratic form and it follows from Remark \ref{rmk:quadraticform} that  $\overline{\Prym_c(\pi)}\in \mathcal{A}_{g-1}^{trop,V}$.

We are finally ready to prove our main theorem.
\begin{thm}\label{thm:prymtorellicontinuous}The tropical Prym--Torelli map $P^{trop}_g$ is continuous. \end{thm}
\begin{proof}
First of all, the Prym variety of a harmonic double cover is well defined. Indeed,
let $\pi:\wt \Gamma_{\tilde w}\to \Gamma_w$ be obtained  from $\pi':\wt \Gamma'_{\tilde w'}\to \Gamma'_{w'}$ by contracting a subgraph $S'\subset \Gamma'$  and 
 from $\pi'':\wt \Gamma''_{\tilde w''}\to \Gamma''_{w''}$ by contracting a subgraph $S''\subset \Gamma''$ and their preimages.  From \cref{continuityofbasis}, given a basis for anti-symmetric cycles $\mathcal B_{\wt T}$ it is possible to choose bases $\mathcal{B}_{\wt T'}$ and $\mathcal{B}_{\wt T''}$ that behave well under edge contractions, namely they both contract to  $\mathcal B_{\wt T}$.

However, some of the cycles can be contracted to the zero cycle.    In that case, the first Betti number of the graph $\tGa_{\tilde w}$ decreases, increasing the total weigth $|\tilde w|$ of the graph and the rank of the lattice $\mathrm{Im} (\Id-\iota_*)$ drops.   The dimension of the integral torus  that defines the Prym variety is not the same anymore, but   the exended Prym variety $\overline{\mathrm{Prym}_c(\pi)}$, namely obtained by adding $|\tilde w|$ dimensions to the torus as in \cref{def:extendedPrym}, has the correct dimension and  corresponds to the limit of the Prym variety that we started with.
Indeed,  let $\pi_i:\wt\Gamma_i\to\Gamma_i$ be a sequence of double covers in the interior of a single cone of $\calR_g$  converging to a dilated double cover $\pi_0:\wt\Gamma_0\to\Gamma_0$ and let $\wt S_i $ and $S_i$ be the subgraphs contracted at each step.
We claim that  
\[\lim_{i\to\infty} \overline{\Prym_c(\pi_i) }= \overline{\Prym_c(\pi_0)}.\]
 Since $\pi_i\overset{i\to\infty}{\longrightarrow}\pi_0$, it follows that $\ell_i(e)$ converges to 0 for all  $e\in\wt S$. Moreover, from \cref{def:extendedPrym} and from \cref{continuityofbasis}, the integration pairing on $\overline{\Prym_c(\pi_0)}$ is given by the integration pairing on any $\overline{\Prym_c(\pi_i)} $ by setting to zero the intersection between anti-symmetric cycles and edges in $\wt S$. The result follows.
\end{proof}

\bibliographystyle{alpha}
\bibliography{Biblio}

\end{document}